\documentclass{article}
\usepackage{mathrsfs}
\usepackage{graphicx}
\usepackage{bbm}
\usepackage{relsize}
\usepackage{enumerate}
\usepackage{amssymb,amsmath,amsthm}
\usepackage{bm}
\usepackage{csquotes}
\usepackage[title]{appendix}
\usepackage{url}
\usepackage{pgf,tikz}
\usetikzlibrary{arrows}
\usetikzlibrary{angles,quotes}
\usetikzlibrary{calc}
\usetikzlibrary{positioning}
\usetikzlibrary{calc,shapes.geometric}
\usetikzlibrary{decorations.pathreplacing}
\usetikzlibrary{cd}
\usetikzlibrary{positioning}
\usepackage{titlesec}
\usepackage{verbatim}
\usepackage{mathrsfs}
\usepackage[letterpaper, portrait, margin=1in]{geometry}
\usepackage{amsmath, amssymb, amsthm}
\usepackage{graphicx}
\usepackage{mathrsfs}
\usepackage{xcolor}
\usepackage{bbm}
\usepackage{mathtools}
\usepackage[linesnumbered,lined,ruled]{algorithm2e}
\usepackage{algpseudocode}
\usepackage{chngcntr}
\usepackage{float,stmaryrd}

\newcommand{\lw}{\textrm{lw}}
\newcommand{\pint}{P_\textrm{int}}

\newtheorem{theorem}{Theorem}

\newtheorem{prop}[theorem]{Proposition}

\newtheorem{cor}[theorem]{Corollary}
\newtheorem{lemma}[theorem]{Lemma}
\newtheorem{defn}[theorem]{Definition}

\numberwithin{theorem}{section}
\numberwithin{equation}{section}

\title{Convex lattice polygons with all lattice points visible}
\author{Ralph Morrison and Ayush Kumar Tewari}

\makeatletter
\def\keywords{\xdef\@thefnmark{}\@footnotetext}
\makeatother

\keywords{2020 \emph{Mathematics Subject Classification.} Primary 52B20; Secondary 52C05, 14T15}%
\keywords{\emph{Key words and phrases.} panoptigons, lattice width, big face graphs}%

\begin{document}

\maketitle

\begin{abstract}
Two lattice points are visible to one another if there exist no other lattice points on the line segment connecting them.  In this paper we study convex lattice polygons that contain a lattice point such that all other lattice points in the polygon are visible from it.  We completely classify such polygons, show that there are finitely many of lattice width greater than $2$, and computationally enumerate them.  As an application of this classification, we prove new obstructions to graphs arising as skeleta of tropical plane curves.
\end{abstract}

\section{Introduction}

A lattice point in $\mathbb{R}^2$ is any point with integer coordinates, and a lattice polygon is any polygon whose vertices are lattice points.  We say that two distinct lattice points $p=(a,b)$ and $q=(c,d)$ are visible to one another if the line segment $\overline{pq}$ contains no lattice points besides $p$ and $q$, or equivalently if $\gcd(a-c,b-d)=1$; by convention we say that any $p$ is visible from itself.  Points visible from the origin $O=(0,0)$ are called visible points, with all other points being called invisible. The properties of visible and invisible points have been subject to a great deal of study over the past century, as surveyed in \cite[\S 10.4]{research_book}.  The question of which structures can appear among visible points, invisible points, or some prescribed combination thereof was  studied in \cite{herzog-stewart}, where it was proved that one can find a copy of any convex lattice polygon (indeed, any arrangement of finitely many lattice points) consisting entirely of invisible points.

\begin{figure}[hbt]
\centering
\includegraphics{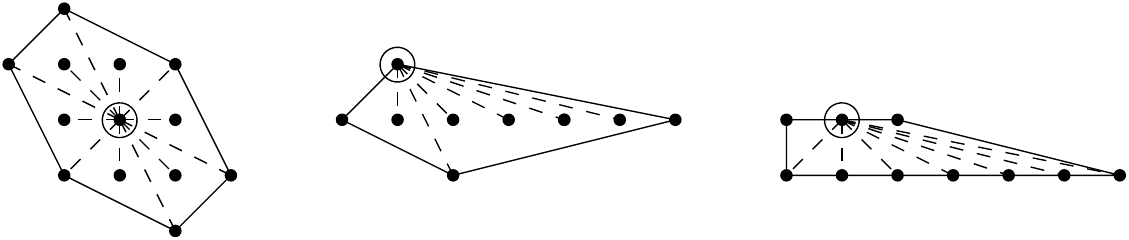}
\caption{Three panoptigons, with a panoptigon point circled and lines of sight illustrated; the middle polygon has a second panoptigon point, namely the bottom vertex}
\label{figure:several_panoptigons}

\end{figure}

In this paper we pose and answer a somewhat complementary question:  which convex lattice polygons including the origin contain only {visible} lattice points?  We define a \emph{panoptigon}\footnote{This name is modeled off of \emph{panopticon}, an architectural design that allows for one position to observe all others.  It comes from the Greek word \emph{panoptes}, meaning ``all seeing''.} to be a convex lattice polygon $P$ containing a lattice point $p$ such that all other lattice points in $P$ are visible from $p$.  We call such a $p$ a \emph{panoptigon point} for $P$.  Thus up to translation, a panoptigon is a convex lattice polygon containing the origin such that every point in $P\cap\mathbb{Z}^2$ is a visible point.  Three panoptigons are pictured in Figure \ref{figure:several_panoptigons}, each with a panoptigon point and its lines of sight highlighted; note that the panoptigon point need not be unique.

One can quickly see that there exist infinitely many panoptigons; for instance, the triangles with vertices at $(0,0)$, $(1,0)$, and $(a,1)$ are panoptigons for any value of $a$.  However, this is not an interesting family of examples since any two of these triangles are \emph{equivalent}, made precisely below.

\begin{defn}
A \emph{unimodular transformation} is an integer linear map $t:\mathbb{R}^2\rightarrow\mathbb{R}^2$ that preserves the integer lattice $\mathbb{Z}^2$; any such map is of the form $t(p)=Ap+b$, where $A$ is a $2\times 2$ integer matrix with determinant $\pm1$ and $b\in\mathbb{Z}^2$ is a translation vector.  We say that two lattice polygons $P$ and $Q$ are \emph{equivalent} if there exists a unimodular triangulation $t$ such that $t(P)=Q$.
 \end{defn}
 
 It turns out that there are infinitely many panoptigons even up to equivalence: note that the triangle with vertices at $(0,0)$, $(0,-1)$, and $(b,-1)$ is a panoptigon for every positive integer $b$, and any two such triangles are pairwise inequivalent since they have different areas.  We can obtain nicer results if we stratify polygons according to the \emph{lattice width} of a polygon $P$, the minimum integer $w$ such that there exists a polygon $P'$ equivalent to $P$ in the horizontal strip $\mathbb{R}\times [0,w]$.  Although there are infinitely many panoptigons of lattice widths $1$ and $2$, we can still classify them completely, as presented in Lemmas \ref{lemma:panoptigon_g0} and \ref{lemma:panoptigon_g1}.  Once we reach lattice width $3$ or more, we obtain the following powerful result.

\begin{theorem}\label{theorem:at_most_13}
Let $P$ be a panoptigon with lattice width  $\lw(P)\geq 3$.  Then $|P\cap\mathbb{Z}^2|\leq 13$.
\end{theorem}

Since there are only finitely many lattice polygons with a fixed number of lattice points up to equivalence \cite[Theorem 2]{lz}, it follows that there are only finitely many panoptigons $P$ with $\lw(P)\geq 3$.  In  Appendix \ref{section:appendix} we detail computations to enumerate all such lattice polygons. This allow us to determine that there exactly $73$  panoptigons of lattice width $3$ or more.  One is the triangle of degree $3$, which has a single interior lattice point; and the other $72$ are non-hyperelliptic, meaning that the convex hull of their interior lattice points is two-dimensional.

As an application of our classification of panoptigons, we prove new results about {tropically planar graphs} \cite{small2017_tpg}.  These are $3$-regular, connected, planar graphs that arise as skeletonized versions of dual graphs of regular, unimodular triangulations of lattice polygons.  We often stratify tropically planar graphs by their first Betti number, also called their genus.  If $G$ is a tropically planar graph arising from a triangulation of a lattice polygon $P$, then the genus of $G$ is equal to the number of interior lattice points of $P$. 

We prove a new criterion for ruling out certain graphs from being tropically planar, notable in that the graphs it applies to are $2$-edge-connected, unlike those ruled out by most existing criteria; this resolves an open question posed in \cite[\S 5]{small2017_tpg}.  We say that a planar graph $G$ is a \emph{big face graph} if for every planar embedding of $G$, there is a bounded face sharing an edge with all other bounded faces.

\begin{theorem}\label{theorem:big_face_graphs}
If $G$ is a big face graph of genus $g\geq 14$, then $G$ is not tropically planar.
\end{theorem}

The idea behind the proof of this theorem is as follows.  If a big face graph $G$ is tropically planar, then it is dual to a regular unimodular triangulation of a lattice polygon $P$.  One of the interior lattice points $p$ of $P$ must be connected to all the other interior lattice points, so that the bounded face dual to $p$ can share an edge with all other bounded faces.  Thus, the convex hull of the interior lattice points of $P$ must be a panoptigon.  If that panoptigon has lattice width $3$ or more, then it can have at most $13$ lattice points, and so $G$ cannot have $g\geq 14$.  

For the case that the lattice width of the interior panoptigon is smaller, we need an understanding of which polygons of lattice width $1$ or $2$ can appear as the interior lattice points of another lattice polygon.  We obtain this in Propositions \ref{prop:lattice_width_3_maximal} and \ref{prop:lattice_width_4_maximal}, and can once again bound the genus of $G$.  In fact, if we are willing to rely on our computational enumeration of all panoptigons with lattice width at least $3$, then we can improve this result to say that big face graphs of genus $g\geq 12$ are not tropically planar. We will see that this bound is sharp.

Our paper is organized as follows.  In Section \ref{section:lattice_polygons} we present background on lattice polygons, including a description of all polygons of lattice width at most $2$.  In Section \ref{section:panoptigons} we classify all panoptigons.  In Section \ref{section:lw3_and_4} we classify all maximal polygons of lattice width $3$ or $4$.  Finally, in Section \ref{section:big_face_graphs} we prove Theorem \ref{theorem:big_face_graphs}.  Our computational results are then summarized in Appendix \ref{section:appendix}.

\medskip

\noindent \textbf{Acknowledgements.}  The authors thank Michael Joswig for many helpful conversations on tropically planar graphs, and for comments on an earlier draft of this paper, as well as two anonymous reviewers for many helpful suggestions and corrections. We would also like to thank Benjamin Lorenz, Andreas Paffenholz and Lars Kastner for helping with  uploading our results to polyDB. Ralph Morrison was supported by the Max Planck Institute for Mathematics in the Sciences, and by the Williams College Class of 1945 World Travel Fellowship. Ayush Kumar Tewari was supported by the Deutsche
Forschungsgemeinschaft (SFB-TRR 195 “Symbolic Tools in Mathematics and their Application”).

\section{Lattice polygons}\label{section:lattice_polygons}

In this section we recall important terminology and results regarding lattice polygons.  This includes the notion of maximal polygons, and of lattice width.
 Throughout we will assume that $P$ is a two-dimensional convex lattice polygon, unless otherwise stated.

The \emph{genus} of a polygon $P$ is the number of lattice points interior to $P$. A key fact is that for fixed $g\geq 1$, there are only finitely many lattice polygons of genus $g$, up to equivalence \cite[Theorem 9]{Castryck2012}.   We refer to the convex hull of the $g$ interior points of $P$ as the \emph{interior polygon of $P$}, denoted $\pint$.  If $\dim(\pint)=2$, we call $P$ \emph{non-hyperelliptic}; if $\dim(\pint)\leq 1$, we call $P$ \emph{hyperelliptic}. This terminology, due to \cite{Castryck2012}, is inspired by hyperelliptic curves, whose Newton polygons have all interior points collinear.  We say a lattice polygon $P$ is a \emph{maximal polygon} if it is maximal with respect to containment among all lattice polygons containing the same set of interior lattice points.

In the case that $P$ is non-hyperelliptic, there is a strong relationship between $P$ and $\pint$.  Let $\tau_1,\ldots,\tau_n$ be the one-dimensional faces of a (two-dimensional) lattice polygon $Q$.  Then $Q$ can be defined as an intersection of half-planes:
\[Q=\bigcap_{i=1}^n\mathcal{H}_{\tau_i},\]
where $\mathcal{H}_{\tau}=\{(x,y)\in\mathbb{R}^2\,|\,a_\tau x+b_\tau y\leq c_\tau\}$ is the set of all points on the same side of the line containing $\tau$ as $Q$.  Without loss of generality, we assume that $a_\tau,b_\tau,c_\tau\in\mathbb{Z}$ with $\gcd(a_\tau,b_\tau)=1$.  With this convention, we define
\[\mathcal{H}^{(-1)}_{\tau}=\{(x,y)\in\mathbb{R}^2\,:\,a_\tau x+b_\tau y\leq c_\tau+1\},\]
and from there define the \emph{relaxed polygon} of $Q$ as
\[Q^{(-1)}:=\bigcap_{i=1}^n\mathcal{H}^{(-1)}_{\tau_i}.\]
We can think of $Q^{(-1)}$ as the polygon we would get by ``moving out'' the edges of $Q$.  It is worth remarking that that $Q^{(-1)}$ need not be a lattice polygon.   We  denote $Q^{(-1)}\cap \mathcal{H}_{\tau_i}^{(-1)}$ as $\tau_i^{(-1)}$.  It is not necessarily the case that $\tau_i^{(-1)}$ is a one-dimensional face of $Q^{(-1)}$; however, if $Q^{(-1)}$ is a lattice polygon, then $Q^{(-1)}\cap \tau_i^{(-1)}$ must contain at least one lattice point, as proved in \cite[Lemma 2.2]{small2017_dim}.  Examples where $Q^{(-1)}$ is not a lattice polygon, and where $Q^{(-1)}$ is a lattice polygon  but an edge has collapsed, are illustrated in Figure \ref{figure:examples_relaxed_polygons}.  There is a very important case when we are guaranteed to have that $Q^{(-1)}$ is a lattice polygon, namely when $Q=\pint$ for some non-hyperelliptic lattice polygon $P$.

\begin{figure}[hbt]
\centering
\includegraphics{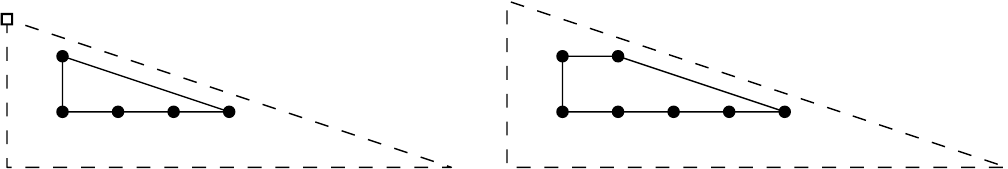}
\caption{Two lattice polygons, one with a relaxed polygon with a non-lattice vertex marked; and one with a collapsed edge in the relaxed (lattice) polygon}
\label{figure:examples_relaxed_polygons}

\end{figure}

\begin{prop}[\cite{Koelman}, \S 2.2]\label{prop:interior_maximal}
Let $P$ be a non-hyperelliptic lattice polygon, with interior polygon $\pint$.  Then $\pint^{(-1)}$ is a lattice polygon containing $P$ whose interior polygon is also $\pint$.  In particular, $\pint^{(-1)}$ is the unique maximal polygon with interior polygon $\pint$.
\end{prop}

If we are given a polygon $Q$ and we wish to know if there exists a lattice polygon $P$ with $\pint=Q$, it therefore suffices to compute the relaxed polygon $Q^{(-1)}$, and to check whether its vertices have integral coordinates.  This might fail because two adjacent edges $\tau_i$ and $\tau_{i+1}$ of $Q$ are relaxed to intersect at a non-integral vertex of $Q^{(-1)}$; we also might have that some $\tau_i^{(-1)}$ is completely lost, which cannot happen when $Q^{(-1)}$ is a lattice polygon by \cite[Lemma 2.2]{small2017_dim}.  Careful consideration of these obstructions will be helpful in classifying the maximal polygons of lattice widths $3$ and $4$ in Section \ref{section:lw3_and_4}.

An important tool in studying lattice polygons is the notion of \emph{lattice width}.  Let $P$ be a non-empty lattice polygon, and let $v=\langle a,b\rangle$ be a lattice direction with $\gcd(a,b)=1$.  The \emph{width of $P$ with respect to $v$} is the smallest integer $d$ for which there exists $m\in \mathbb{Z}$ such that the strip \[m\leq ay-bx\leq m+d\] contains $P$.  We denote this $d$ as $w(P,v)$.  The \emph{lattice width of $P$} is the minimal width over all possible choices of $v$:
\[\textrm{lw}(P)=\min_vw(P,v).\]
Any $v$ which achieves this minimum is called a \emph{lattice width direction for $P$}. 
Equivalently, $\textrm{lw}(P)$ is the smallest $d$ such that there exists a lattice polygon $P'$ equivalent to $P$ with $P'\subset \mathbb{R}\times[0,d]$.

We recall the following result connecting the lattice widths of a polygon and its interior polygon.  Let $T_d=\textrm{conv}((0,0),(d,0),(0,d))$ denote the standard triangle of degree $d$. 

\begin{lemma}[Theorem 4 in \cite{castryckcools-gonalities}]\label{lemma:lw_facts}
For a lattice polygon $P$ we have $\textrm{lw}(P)=\lw(P_{\textrm{int}})+2$, unless $P$ is equivalent to $T_d$ for some $d\geq 2$, in which case $\textrm{lw}(P)=\textrm{lw}(P_{\textrm{int}})+3=d$.

\end{lemma}

The following result tells us precisely which polygons have lattice width $1$ or $2$.  It is a slight reworking due to of a result due to \cite{Koelman}, also presented in \cite[Theorem 10]{Castryck2012}

\begin{theorem}\label{theorem:lw_012}
Let $P$ be a two-dimensional lattice polygon. If $\textrm{lw}(P)=1$, then $P$ is equivalent to
\[T_{a,b}:=\textrm{conv}((0,0),(0,1),(a,1),(b,0))\]
for some $a,b\in\mathbb{Z}$ with $0\leq a\leq b$ and $b\geq 1$.

If $\textrm{lw}(P)=2$, then up to equivalence either $P=T_2$; or $g(P)=1$ and $P\neq T_3$ (all such polygons are illustrated in Figure \ref{figure:g1_lw2}); or $g(P)\geq 2$.  In the latter case we have $\frac{1}{6}(g+3)(2g^2+15g+16)$ polygons, sorted into three types:
\begin{itemize}
    \item Type 1:
    \[\includegraphics{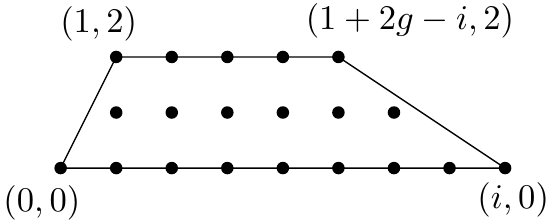}\]
    where $g\leq i\leq 2g$.
    \item Type 2:
        \[\includegraphics{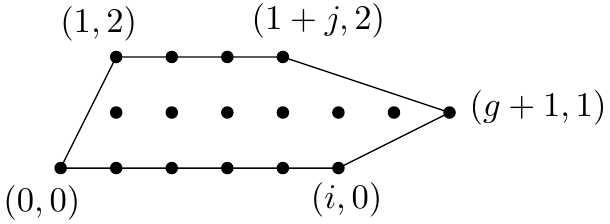}\]
where $0\leq i\leq g$ and $0\leq j\leq i$; or $g<i\leq 2g+1$ and $0\leq j\leq 2g-i+1$
    
    \item Type 3:
        \[\includegraphics{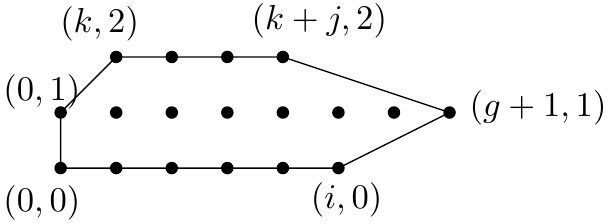}\]
where $0\leq k\leq g+1$ and $0\leq i\leq g+1-k$ and $0\leq j\leq i$; or $0\leq k\leq g+1$ and $g+1-k<i\leq 2g+2-2k$ and $0\leq j\leq 2g-i-2k+1$
\end{itemize}

\begin{figure}[hbt]
\centering
\includegraphics{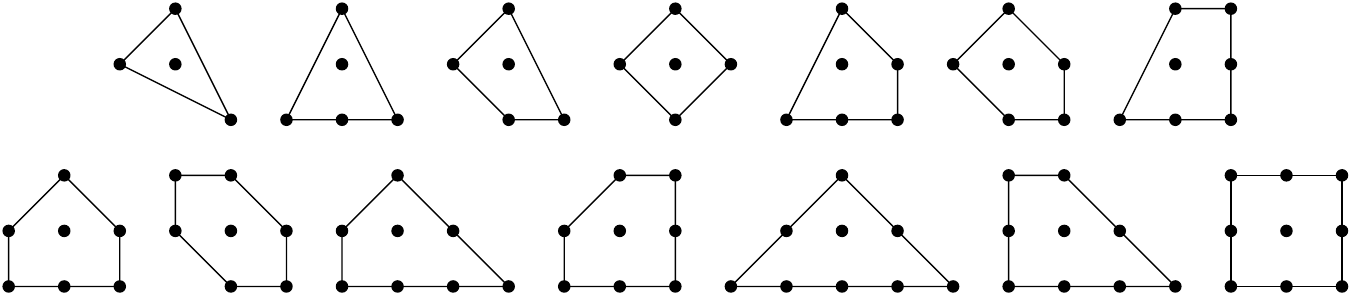}
\caption{The $14$ genus $1$ polygons with lattice width $2$}
\label{figure:g1_lw2}

\end{figure}

\end{theorem}

\begin{proof}The classification proved in \cite{Koelman} was similar, except with polygons sorted by genus ($g=0$, $g=1$, and $g\geq 2$ with all interior lattice points collinear) rather than by lattice width.  We can translate their work into the desired result as follows.

For $\lw(P)=1$, we know $P$ has no interior lattice points, so $g=0$; all polygons of genus $0$ besides $T_2$ have lattice width $1$.  By \cite{Koelman} all genus $0$ polygons besides $T_2$ are equivalent to $T_{a,b}$
for some $a,b\in\mathbb{Z}$ with $0\leq a\leq b$ and $b\geq 1$.

For $\lw(P)=2$, we deal with the three cases of $g=0$, $g=1$, and $g\geq 2$.  If $g=0$, then the only polygon of lattice width $2$ is $T_2$.  If $P$ is a polygon with genus $g=1$, then by Lemma \ref{lemma:lw_facts} we know that $\lw(P)=\lw(\pint)+2=0+2=2$ unless $P$ is equivalent to $T_d$ for some $d$.  The only value of $d$ such that $T_d$ has genus $1$ is $d=3$, so every genus $1$ polygon except $T_3$ has lattice width $2$.

Finally, suppose $P$ is a polygon of lattice width $2$ and genus $g\geq 2$.  Since $\lw(T_d)=d$ and $g(T_2)=0$, we know $P\neq T_d$ for any $d$, and so $\lw(\pint)=\lw(P)-2=2-2=0$.  It follows that all the $g$ interior lattice points of $P$ must be collinear, and so $P$ is hyperelliptic. Conversely, if $P$ is a hyperelliptic polygon of genus $g\geq 2$, by definition the interior polygon $\pint$ has lattice width $0$. Since  no triangle $T_d$ has genus $g\geq 2$ with all its interior points collinear we may apply Lemma \ref{lemma:lw_facts} to conclude that $\lw(P)=\lw(\pint)+2=2$.  This means that for polygons of genus $g\geq 2$, being hyperelliptic is equivalent to having lattice width $2$.  Combined with the classification of hyperelliptic polygons in \cite{Koelman}, this completes the proof.
\end{proof}

A counterpart of lattice width is \emph{lattice diameter}.  Following \cite{bf}, the lattice diameter $\ell(P)$ is the length of the longest lattice line segment contained in the polygon $P$:
\[\ell(P)=\max\{|L\cap P\cap\mathbb{Z}^2|-1\,:\,\textrm{$L$ is a line}\}.\]
We define a \emph{lattice diameter direction} $\left<a,b\right>$ to be one such that there exists a line $L$ with slope vector $\left<a,b\right>$ with $|L\cap P\cap\mathbb{Z}^2|-1=\ell(P)$.  
We remark that there exist other works where lattice diameter is defined as the largest number of collinear lattice points in the polygon $P$ \cite{alarcon}; this is simply one more than the convention we set above.  The following result relates $\ell(P)$ to $\lw(P)$.

\begin{theorem}[\cite{bf}, Theorem 3]
We have $\lw(P)\leq \lfloor\frac{4}{3}\ell(P)\rfloor+1$.
\end{theorem}

We now present background material on triangulations and tropical curves.  Assume for the remainder of the section that $P$ is a lattice polygon of genus $g\geq 2$.  A \emph{unimodular triangulation} of $P$ is a subdivision of $P$ into lattice triangles of area $\frac{1}{2}$ each.  Such a triangulation $\Delta$ is called \emph{regular} if there exists a height function $\omega:P\cap\mathbb{Z}^2\rightarrow \mathbb{R}$ inducing $\Delta$.  This means that $\Delta$ is the projection of the lower convex hull of the image of $\omega$ back onto $P$.  See \cite{triangulations} for details on regular triangulations.

Given a regular unimodular triangulation $\Delta$ of a lattice polygon $P$, we can consider the \emph{weak dual graph} of $\Delta$, which consists of $1$ vertex for each elementary triangle, with two vertices connected if and only if the corresponding triangles share an edge.  Each vertex in this graph has degree $1$, $2$, or $3$, depending on how many edges the corresponding triangle has on the boundary of $P$.  We transform the weak dual graph of $\Delta$ into a $3$-regular graph as follows:  first, iteratively delete any $1$-valent vertices and their attached edges.  This will yield a graph with all vertices of degree $2$ or $3$.  Remove each degree $2$ vertex by concatenating the two edges incident to it.  Since we have assumed that $g(P)\geq 2$, the end result is a $3$-regular graph $G$ (with loops and parallel edges allowed).  We call $G$ the \emph{skeleton} associated to $\Delta$.  Any $G$ that arises from such a procedure is called a \emph{tropically planar graph}.  An example of a regular unimodular triangulation, the weak dual graph, and the tropically planar skeleton are pictured in Figure \ref{figure:tropically_planar_full}.  Note that there is a one-to-one correspondence between the interior lattice points of $P$ and the bounded faces of $G$ in this embedding, where two faces of $G$ share an edge if and only if the corresponding interior lattice points are connected by an edge in $\Delta$.

\begin{figure}[hbt]
\centering
\includegraphics[scale=0.6]{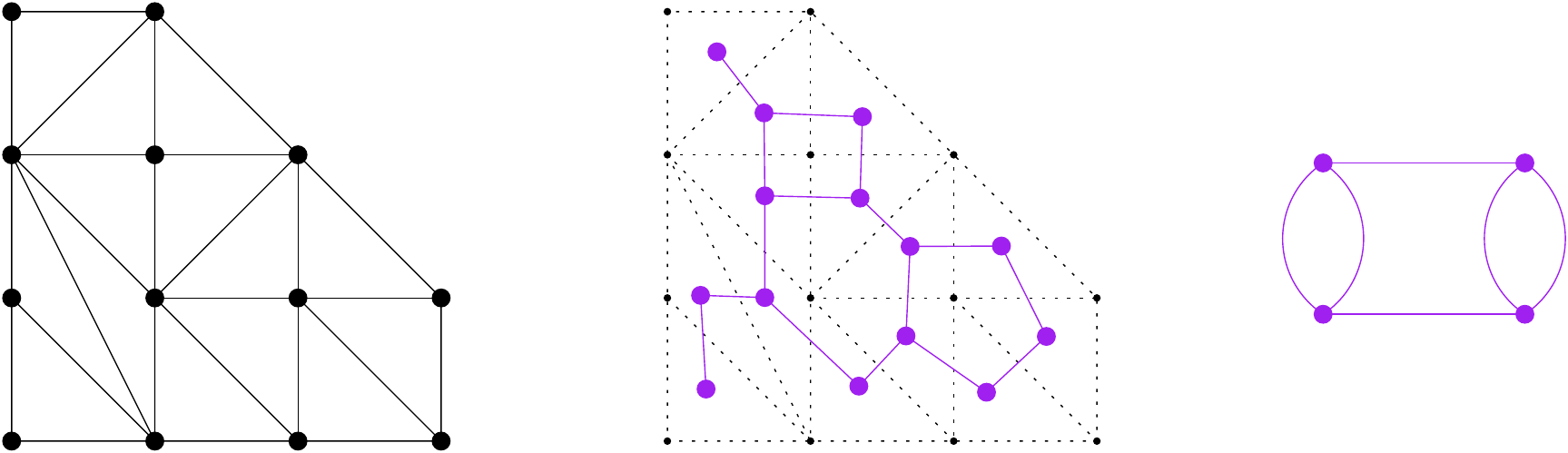}
\caption{A regular unimodular triangulation of a polygon, the weak dual graph of the triangulation, and the corresponding tropically planar skeleton}
\label{figure:tropically_planar_full}

\end{figure}

It is worth remarking that we could still construct a graph $G$ from a non-regular triangulation.  The reason that we insist that $\Delta$ is regular is so that the graph $G$ appears as a subset of a smooth tropical plane curve, which is a balanced $1$-dimensional polyhedral complex that is dual to a regular unimodular triangulation of a lattice polygon; see \cite{ms}. (Indeed, the regularity is necessary if we wish to endow a skeleton with the structure of a \emph{metric graph}, with lengths assigned to its edges, as explored in \cite{BJMS} and \cite{small2017_dim}.) Most of the results that we prove in this paper, and that we recall for the remainder of this section, also hold if we expand to graphs that arise as dual skeleta of \emph{any} unimodular triangulation of a lattice polygon.

The first Betti number of a tropically planar graph, also known as its genus\footnote{This terminology comes from \cite{bn} and is motivated by algebraic geometry; it is unrelated to the notion of graph genus defined in terms of embeddings on surfaces.  The first Betti number of a graph is also sometimes called its \emph{cyclomatic number}.}, is equal to the number of interior lattice points of the lattice polygon $P$ giving rise to it.  It is also equal to the number of bounded faces in any planar embedding of the graph.  A systematic method of computing all tropically planar graphs of genus $g$ was designed and implemented in \cite{BJMS} for $g\leq 5$. The algorithm is brute-force, and works by considering all maximal lattice polygons of genus $g$, finding all regular unimodular triangulations of them, and computing the dual skeleta.  These computations were pushed up to $g= 7$ in \cite{small2017_tpg}. In general there is no known method of checking whether an arbitrary graph is tropically planar short of this massive computation.

A fruitful direction in the study of tropically planar graphs has been finding properties or patterns that are forbidden in such graphs, so as to quickly rule out particular examples.  Since the graph before skeletonization is dual to a unimodular triangulation of a polygon,  any tropically planar graph is $3$-regular, connected, and planar.  Several additional constraints are summarized in the following result.

\begin{theorem}[\cite{cdmy}, Proposition 4.1; \cite{small2017_tpg}, Theorem 3.4; \cite{joswigtewari}, Theorems 10 and 14]

Suppose that $G$ is a $3$-regular graph of genus $g$ of one of the forms illustrated in Figure \ref{figure:forbidden_patterns}, where each gray box represents a subgraph of genus at least $1$.  If $G$ is tropically planar, then it must have either the third or fourth forms, with $g= 4$ for the third form and $g\leq 5$ in the fourth form.  In particular, if $g\geq 6$, then $G$ is not tropically planar.

\end{theorem}

\begin{figure}[hbt]
\centering
\includegraphics[scale=0.7]{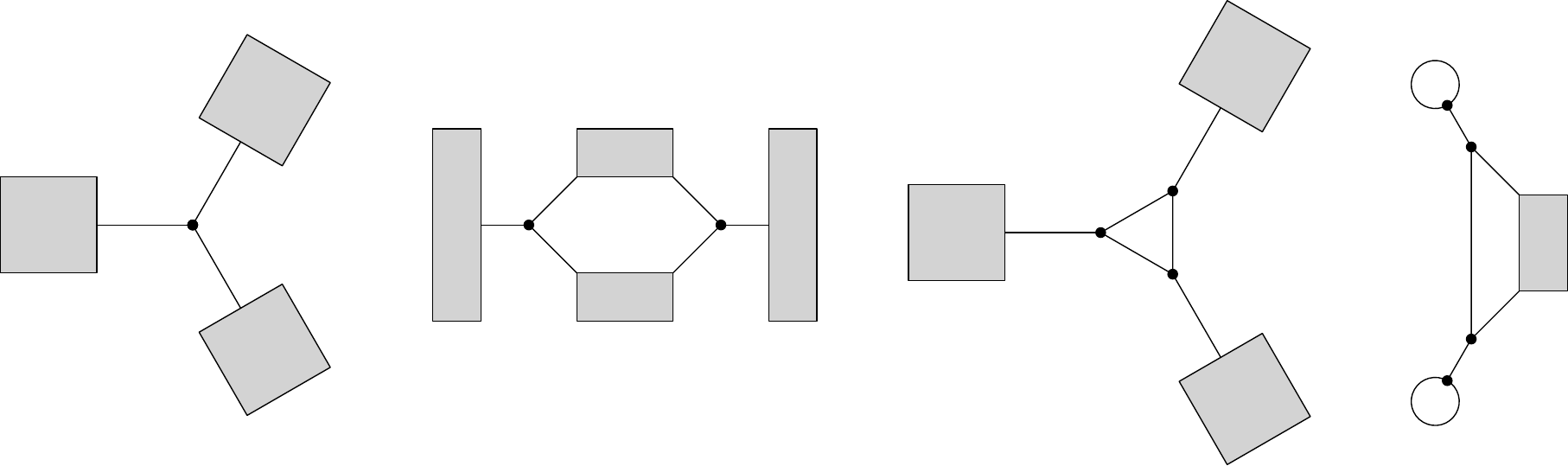}
\caption{Forbidden patterns in tropically planar graphs of genus $g\geq 6$}
\label{figure:forbidden_patterns}

\end{figure}

The proofs of these results all use the following observation:  any cut-edge in a tropically planar graph must arise from a split in the dual unimodular triangulation that divides the polygon into two polygons of positive genus. From there, one argues that collections of such splits cannot appear in lattice polygons in ways that would give rise to graphs of the pictured forms.  For planar graphs that are $2$-edge-connected and thus have no cut-edges, the only known general criterion to rule out tropical planarity is the notion of crowdedness \cite{morrisonhyp}.  However, crowded graphs are ones that cannot be dual to \emph{any} triangulation of \emph{any} point set in $\mathbb{R}^2$, regardless of whether or not the point set comes from a convex lattice polygon; thus it is not especially interesting that crowded graphs are not tropically planar.  In Section \ref{section:big_face_graphs} we will find a family of $2$-edge-connected, $3$-regular  planar graphs that are not crowded but are still not tropically planar, the first known such examples.

\section{A classification of all panoptigons}\label{section:panoptigons}

Let $P$ be a convex lattice polygon.  Recall from the introduction that $P$ is a {panoptigon} if there is lattice point $p\in P\cap\mathbb{Z}^2$  such that every other point in $P\cap\mathbb{Z}^2$ is visible from $p$.  In this section we will classify all panoptigons, stratified by a combination of genus and lattice width.  We begin with the panoptigons of genus $0$.

\begin{lemma}\label{lemma:panoptigon_g0}
Let $P$ be a panoptigon of genus $0$.  Then $P$ is one of the following polygons, up to lattice equivalence:
\[\includegraphics[]{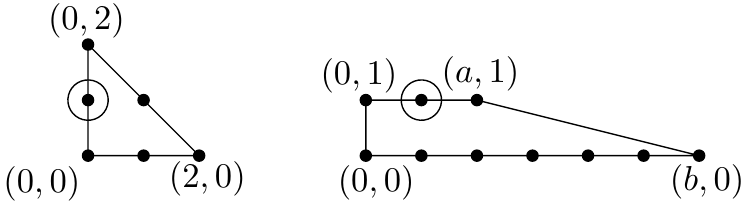}\]
where $0\leq a\leq\min\{2,b\}$.
\end{lemma}

\begin{proof}
By \cite{Koelman}, any genus $0$ polygon is equivalent either to the triangle $T_2$, or to the (possibly degenerate) trapezoid $T_{a,b}$ where $0\leq a\leq b$ and $1\leq b$.  The triangle of degree $2$ is a panoptigon, as any non-vertex lattice point can see every other lattice point.  For $T_{a,b}$, we note that if $a\geq 3$ then the polygon is not a panoptigon:  each lattice point $p$ is on a row with at least $3$ other lattice points, not all of which can be visible from $p$ since the $4$ (or more) points in that row are collinear.  However, if $a\leq 2$, then a point $p$ can be chosen on the top row that can see the other $a$ points on the top row, as well as all points on the bottom row.  Thus $T_{a,b}$ is a panoptigon if and only if $a\leq 2$.
\end{proof}

For polygons with exactly one interior lattice point, there is no obstruction to being a panoptigon.

\begin{lemma}\label{lemma:panoptigon_g1}
If $P$ is a polygon of genus $1$, then $P$ is a panoptigon.
\end{lemma}

\begin{proof}  Let $p$ be the unique interior lattice point of $P$, and let $q$ be any other lattice point of $P$.  Since $g(P)=1$, the point $q$ must be on the boundary.  By convexity, the line segment $\overline{pq}$ must have its relative interior contained in the interior of the polygon, and so the line segment does not intersect $\partial P$ outside of $q$.  Since $p$ is the only interior lattice point, we have that the only lattice points of $\overline{pq}$ are its endpoints.   It follows that $q$ is visible from $p$ for all $q\in P\cap\mathbb{Z}^2-\{p\}$.  We conclude that $P$ is a panoptigon with panoptigon point $p$.
\end{proof}

We now consider hyperelliptic polygons of genus $g\geq 2$.  We will characterize precisely which of these are panoptigons based on the classification of them in Theorem \ref{theorem:lw_012} into Types 1, 2, and 3.  Any hyperelliptic polygon can be put into one of these forms in the horizontal strip $\mathbb{R}\times[0,2]$; thus we may say a lattice point $(a,b)$ of such a polygon is at height $b$, where every point is either at height $0$, height $1$, or height $2$.

\begin{lemma}\label{lemma:hyperelliptic_panoptigon}
Let $P$ be a hyperelliptic polygon of genus $g\geq 2$,  transformed so that it of one of the forms presented in Theorem \ref{theorem:lw_012}.  Then $P$ is a panoptigon if and only if
\begin{itemize}
    \item $P$ is of Type 1, with $g\leq 3$; or
    \item $P$ is of Type 2, either with $g\leq 2$, or with $j=0$ and $0\leq i\leq 1$; or
    \item  $P$ is of Type 3, either with $j=0$ and $i\leq 2$, with $k$ odd if $i=0$ and $k$ even if $i=2$; or with $i=0$ and $j\leq 2$, and $k$ odd if $j=0$ and $k$ even if $j=2$.
\end{itemize}
\end{lemma}

For the reader's convenience we recall the polygons of Types 1, 2, and 3 in Figure \ref{figure:hyp_all_types}.

\begin{figure}[hbt]
\centering
\includegraphics[scale=0.8]{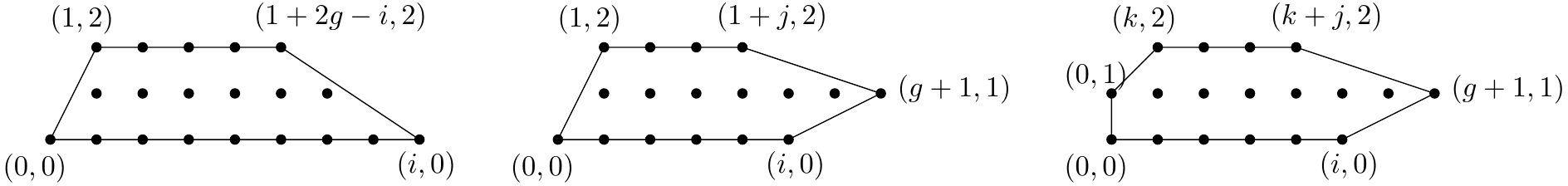}
\caption{Hyperelliptic polygons of Types 1, 2, and 3}
\label{figure:hyp_all_types}

\end{figure}

\begin{proof}
We start by making the following observations.  If $p=(a,b)$ is a panoptigon point for a hyperelliptic polygon $P$, then there must be at most $3$ points at height $b$; and if there are exactly $3$, then $p$ must be the middle such point.  We also make several remarks in the case that $b\in\{0,2\}$.  There are no obstructions to a point at height $b$ seeing a point at height $1$, so we will not concern ourselves with this.  Choose $b'\in\{0,2\}$ distinct from $b$, and suppose height $b'$ has $2$ or more lattice points; then two of those points have the form $q=(a,b')$ and $q'=(a+1,b')$.  We claim that $p$ cannot view both $q$ and $q'$. Writing $p=(a,b)$, the midpoints of the line segments $\overline{pq}$ and $\overline{pq'}$ have coordinates $\left(\frac{a+a'}{2},1\right)$ and $\left(\frac{a+a'+1}{2},1\right)$, respectively.  Exactly one of $\frac{a+a'}{2}$ and $\frac{a+a'+1}{2}$ is an integer, meaning that either $q$ or $q'$ is not visible from $p$.  So, if $p=(a,b)$ is a panoptigon point at height $b\in\{0,2\}$, there must be exactly one lattice point $q=(a',b')$ at height $b'\in \{0,2\}$ with $b'\neq b$; moreover, we  must have that $a-a'$ is odd. 

We are ready to determine the possibilities for a hyperelliptic panoptigon $P$ of genus $g\geq 2$, sorted by type.

\begin{itemize}
    \item Let $P$ be a hyperelliptic polygon of Type 1.  If $g\leq 3$, then we may choose $p=(a,1)$ that can see every other point at height $1$, as well as all points at heights $0$ and $2$; in this case $P$ is a panoptigon.  If $g\geq 4$, then there are at least $4$ points at height $1$.  Moreover, the number of points at height $0$ is $i+1$ where $g\leq i\leq 2g$, and we have $i+1\geq 5$ since $g\geq 4$.  Thus it is impossible to have at most $3$ points at one height and $1$ at another.  This means that for $g\geq 4$, $P$ cannot be a panoptigon
    \item Let $P$ be a hyperelliptic polygon of Type 2.  If $g=2$, then $P$ has exactly three points at height $1$, and we can choose the middle point as a panoptigon point.  Now assume $g\geq 3$; we cannot choose a panoptigon point at height $1$, since there are $g+1\geq 4$ points at that height.  To avoid having $4$ points on both the top and bottom rows we need  $0\leq i\leq g$ and $0\leq j\leq i$; and one of $i$ and $j$ must be $0$, so we need $j=0$ since $j\leq i$.  From there we need at most $3$ lattice points on the bottom row, so $0\leq i\leq 2$. If $i=2$, then the only possible panopticon point is the middle one on the bottom row, namely $(1,0)$; but this point cannot see $(1,2)$, a contradiction. Thus $0\leq i\leq 1$; note that in either case $(0,0)$ can serve as a panoptigon point.
    \item  Finally, let $P$ be a hyperelliptic polygon of Type 3.  We cannot have a panoptigon point at height $1$, since there are at least $g+2\geq 4$ points at that height.  If there is a panoptigon point at height $0$,  then we must have at most $3$ points at height $0$ and exactly one point at height $2$; that is, we must have $j=0$ and $i\leq 2$.  Moreover, we need to verify that way may choose a panoptigon point at height $0$ that can see the unique point at height $2$; this can always be done if $i=1$, but if $j=0$ then we need  $k$ odd (the only possible panoptigon point is then $(0,0)$), and if $j=2$ we need $k$ even (the only possible panoptigon point is then $(1,0)$).  A similar argument shows that we can choose a panoptigon point at height $2$ if and only if $i=0$ and $j\leq 2$, with $k$ odd if $j=0$ and $k$ even if $j=2$.
\end{itemize}
\end{proof}
As with the lattice width $1$ panoptigons, we find infinitely many lattice width $2$ panoptigons, namely those of Type 2 with $j=0$ and $0\leq i\leq 1$, and those of Type 3.

We have now classified all hyperelliptic panoptigons, and have found that there are infinitely many of lattice width $1$ and infinitely many of lattice width $2$.  Our last step is to understand non-hyperelliptic panoptigons; with the exception of the triangle $T_3$, this is equivalent to panoptigons of lattice width $3$ or more.  We are now ready to prove that the total number of lattice points of such a panoptigon is at most $13$.

\begin{proof}[Proof of Theorem \ref{theorem:at_most_13}]  
Let us consider the lattice diameter $\ell(P)$ of $P$.  We know by \cite[Theorem 1]{alarcon} that $|P\cap\mathbb{Z}^2|\leq (\ell(P)+1)^2$, so if $\ell(P)\leq 2$ we have $|P\cap\mathbb{Z}^2|\leq 9$.  Thus we may assume $\ell(P)\geq 3$.

Perform an $\textrm{SL}_2(\mathbb{Z})$ transformation so that $\langle 1,0\rangle$ is a lattice diameter direction for $P$, and translate the polygon so that the origin $O=(0,0)$ is a panoptigon point.  Thus $P\cap\mathbb{Z}^2$ consists of $O$ and a collection of visible points.

Since $\ell(P)\geq3$ and $\langle 1,0\rangle$ is a lattice diameter direction, we know that the polygon $P$ must contain $4$ lattice points of the form $(a,b)$, $(a+1,b)$, $(a+2,b)$, and $(a+3,b)$.  We claim that $b\in\{-1,1\}$.  Certainly $b\neq 0$, since there are only three such points allowed in $P$:  $(0,0)$ and $(\pm 1, 0)$.  We also know that $b$ cannot be even:  any set $\mathbb{Z}\times \{2k\}$ has every second point invisible from the origin.

Suppose for the sake of contradiction that the points $(a,b)$, $(a+1,b)$, $(a+2,b)$, and $(a+3,b)$ are in $P$ with $b$ odd and $b\geq 3$ (a symmetric argument will hold for $b\leq -3$).  Consider the triangle $T=\textrm{conv}(O,(a,b),\ldots,(a+3,b))$.  By convexity, $T\subset P$.  
Consider the line segment $T\cap L$, where $L$ is the line defined by $y=b-1$.  The length of this line segment is $3-\frac{1}{b}$, and since $b\geq 3$ this is strictly greater than $2$.  Any line segment of length $2$ at height $b-1$ will intersect at least two lattice points.  But since $b-1$ is even and $b-1\geq 2$, at least one of these lattice points is not visible from $O$.  Such a lattice point must be contained in $T$, and therefore in $P$, a contradiction.  Thus we have that $b=\pm 1$.

Rotating our polygon $180^\circ$ degrees if necessary, we may assume that $b=-1$, so that the points $(a,-1),\ldots,(a+3,-1)$ are contained in $P$.  It is possible that the number $k$ of lattice points on the line defined by $y=-1$ is more than $4$; up to relabelling, we may assume that $(a,-1),\ldots,(a+k-1,-1)$ are lattice points in $P$ while $(a-1,-1)$ and $(a+k,-1)$ are not, where $k\geq 4$.  Applying a shearing transformation $\left(\begin{matrix}1&a+1\\0&1\end{matrix}\right)$, we may further assume that the points at height $-1$ are precisely $(-1,-1),\ldots,(k-2,-1)$.

We will now make a series of arguments that rule out many lattice points from being contained in $P$.  The end result of these constraints is pictured in Figure \ref{figure:ruling_out_points}, with points labelled by the argument that rules them out.

\begin{itemize}
\item[(i)] The polygon $P$ has (regular) width at least $3$ at height $-1$, and width strictly smaller than $2$ at heights $2$ and $-2$, since it cannot contain two consecutive lattice points at those heights.  It follows from convexity that the width of the polygon is strictly smaller than $1$ at height $-3$, and that the polygon cannot have any lattice points at all at height $-4$.  It also follows that the polygon cannot have a nonnegative width at height $8$.  Thus every lattice point $(x,y)$ in the polygon satisfies $-3\leq y\leq 7$.

\item[(ii)]  We can further restrict the possible heights by showing that there can be no lattice points at height $-3$.  Suppose there were such a point $(x,-3)$ in $P$.  Consider the triangle $\textrm{conv}((x,-3),(-1,-1),(2,-1))$.  This triangle has area $3$, so by Pick's Theorem \cite{pick} satisfies $3=g+\frac{b}{2}-1$, or $4=g+\frac{b}{2}$, where $g$ and $b$ are the number of interior lattice points an boundary lattice points of the triangle, respectively.  The $4$ lattice points at height $-1$ contribute $2$ to this sum, and the one lattice point at height $-3$ contributes $\frac{1}{2}$ to this sum, meaning that the lattice points at height $-2$ contribute $\frac{3}{2}$ to this sum.  It follows that there must be at least two lattice points at height $-2$; but this is a contradiction, since at least one of these points will be invisible from $O$.  We conclude that $P$ cannot contain a lattice point of the form $(x,-3)$, and thus $y\geq -2$ for all lattice points $(x,y)\in P$.

\item[(iii)]  We know that the lattice point $(-2,0)$ is not in $P$ since it is not visible from $O$.  If there is any lattice point of the form $(x,y)$ with $y\geq 1$ and $y\leq -x-2$, then the triangle $\textrm{conv}(O,(-1,-1),(x,y))$ will contain $(-2,0)$.  Thus no such lattice point $(x,y)$ can exist in $P$.

\item[(iv)] No point of the form $(x,y)$ with $x\geq 2$ and $y\geq 0$ may appear in $P$:  this would force the point $(2,0)$ to appear, as it would lie in the triangle $\textrm{conv}(O,(2,-1),(x,y))$.

\item[(v)]  There are now only finitely many allowed lattice points $(x,y)$ with $y\geq 1$, namely those with $-y-1\leq x\leq 2$ and $1\leq y\leq 7$. For each such point, we  consider the triangle $\textrm{conv}((x,y),(-1,-1),(-1,3))$.   We claim that only the $13$ choices of $(x,y)$ pictured in Figure \ref{figure:ruling_out_points}.
 that do not introduce a forbidden point.  To see this, we note that the points $(0,2)$, $(-2,2)$ and $(-2,4)$ are all forbidden.  The point $(0,2)$ rules out $(x,y)$ with $x=1$ and $y\geq 5$; with $x=0$ and $y\geq 2$; with $x=-1$ and $y\geq 4$; ant with $x=-2$ and $y\geq 5$.  For $x=-2$, the points $(-2,2)$ and $(-2,4)$ are already ruled out.  For all remaining points with $x\leq -3$,  every point besides $(-3,2)$, $(-4,3)$, and $(-5,3)$ introduces the point $(-2,2)$ or $(-2,4)$ or both.  This establishes our claim.

\item[(vi)]  By assumption, we know there are no lattice points of the form $(x,-1)$ where $x\leq -2$.  It follows that there are also no lattice points of the form $(x,-2)$ where $x\leq -4$, since $(-1,-2)$ would lie in the convex hull of such a point with $O$ and $(2,-1)$.

\item[(vii)]  We will now use the fact that we have assumed that $P$ satisfies $\lw(P)\geq 3$.  We cannot have that $P$ is contained in the strip $-2\leq y\leq 0$, so there must be at least one point $(x,y)$ with $y\geq 1$.  If there is a point of the form $(x',-1)$ with $x'\geq 6$, then we would have that  $\textrm{conv}((x,y),(x',-1),(-1,-1))$ contains the point $(2,0)$, which is invisible.  Thus we can only have points $(x',-1)$ if $-1\leq x\leq 5$.  A similar argument shows that $P$ can only contain a point $(x,-2)$ if $x$ is odd with $-3\leq x\leq 9$.
\end{itemize}

\begin{figure}[hbt]
\centering
\includegraphics{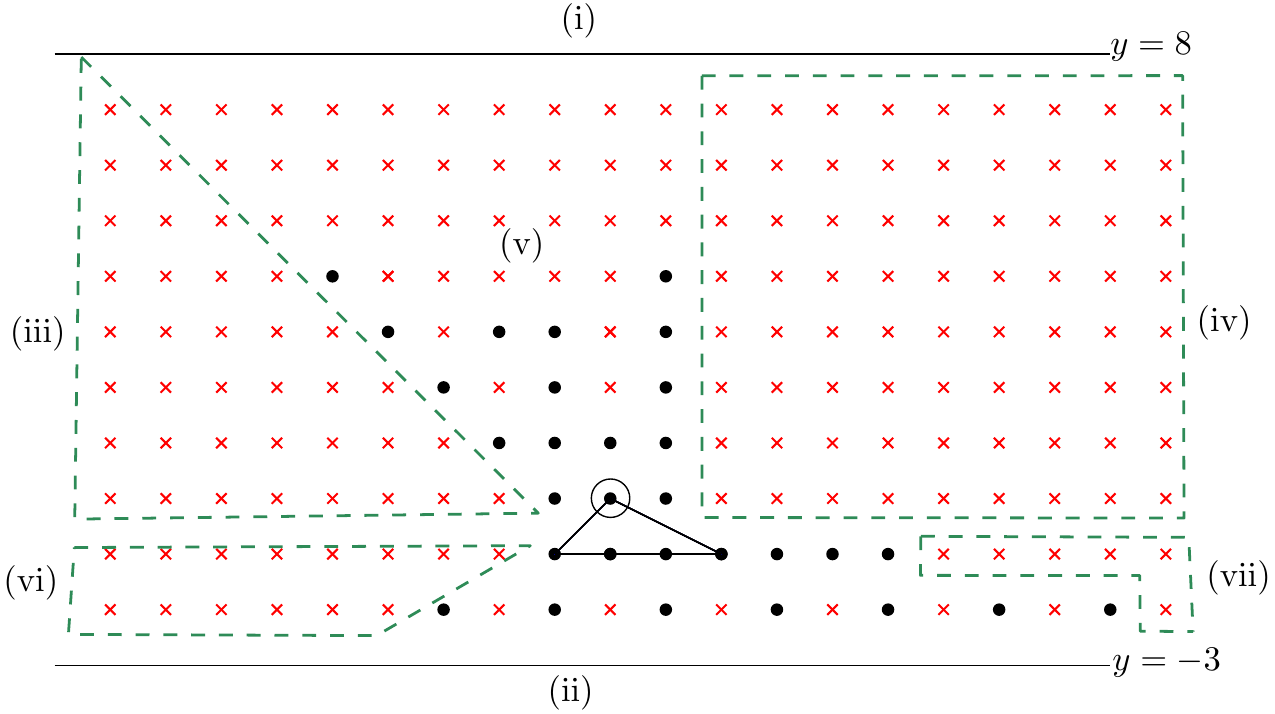}
\caption{Possible lattice points in $P$, with impossible points labelled by the argument ruling them out}
\label{figure:ruling_out_points}

\end{figure}

We have now narrowed the possible lattice points in our polygon down to the $30$ lattice points in Figure \ref{figure:ruling_out_points}, five of which we know appear in $P$.  For every such point $(x,y)$, there does indeed exist a polygon $P$ with $\lw(P)\geq 3$ containing $(x,y)$ as well as the five prescribed points such that $P\cap\mathbb{Z}^2$ is a subset of the $30$ allowed points, so we cannot narrow down any further.

One way to finish the proof is by use of a computer to determine all possible subsets of the $25$ points that can be added to our initial $5$ points to yield a polygon of lattice width at least $3$; we would then simply check the largest number of lattice points.  We have carried out this computation, and present the results in Appendix \ref{section:appendix}.  We also present the following argument, which will complete our proof without needing to rely on a computer.

First we  split into four cases, depending on the number $k$ of lattice points at height $-1$:  $4$, $5$, $6$, or $7$.  When there are more than $4$, we can eliminate more of the candidate points $(x,y)$ with $y\geq 1$ or $y=-2$; the sets of allowable points in these four cases are illustrated in Figure \ref{figure:four_cases_allowed_points}.  In each case we will argue that our polygon $P$ has at most $13$ lattice points.

\begin{figure}[hbt]
\centering
\includegraphics[scale=0.8]{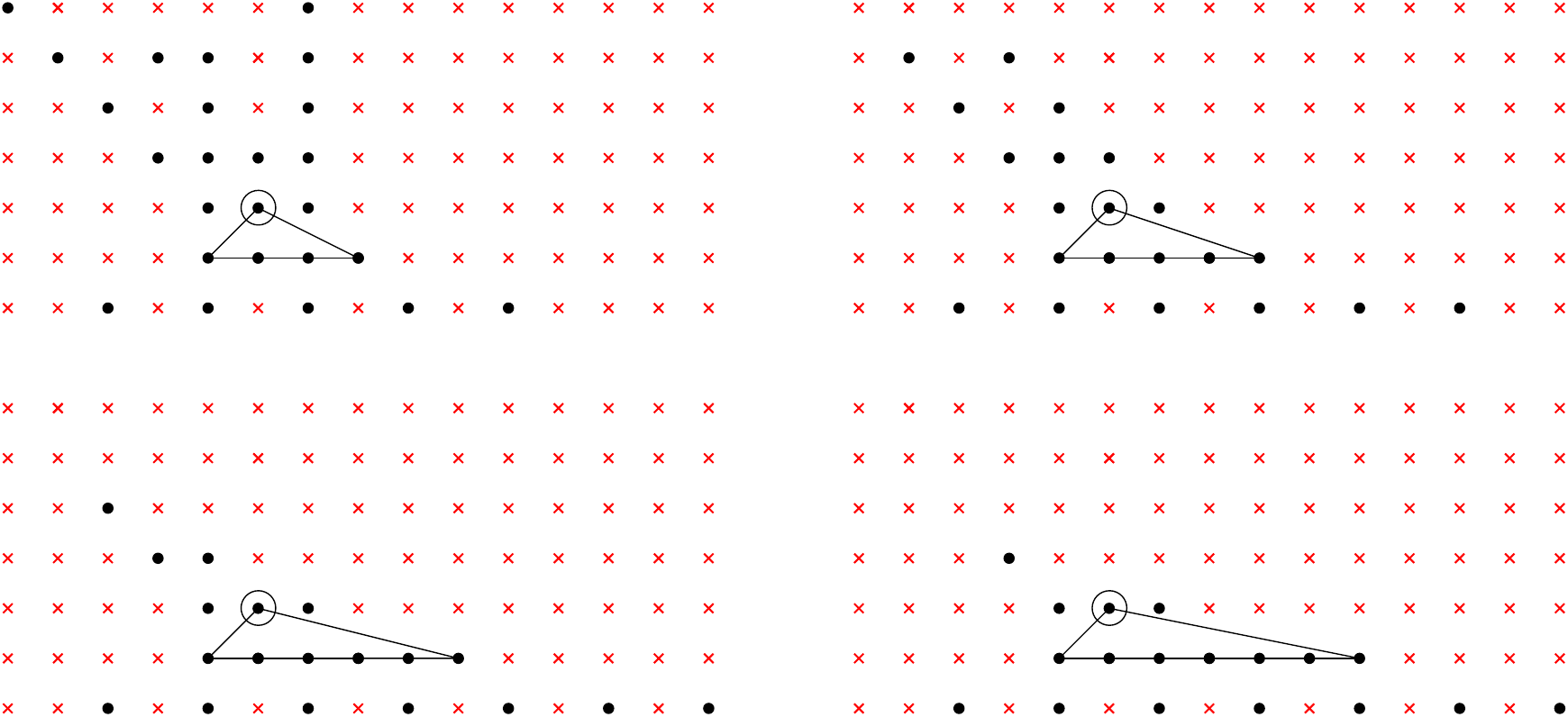}
\caption{Narrowing down possible points depending on the number of points at height $-1$}
\label{figure:four_cases_allowed_points}

\end{figure}

\begin{itemize}
    \item Suppose $k=4$.  There are $20$ possible points at height $-1$ or above; since there is at most one point at height $-2$, it suffices to show that we can fit no more than $12$ lattice points at height $-1$ or above into a lattice polygon.
    
    First suppose the point $(-5,4)$ is in $P$.  This eliminates $9$ possible points from appearing in $P$, yielding at most $20-9+1=12$ lattice points total in $P$. Leaving out $(-5,4)$ but including $(-4,3)$ similarly eliminates $9$ possible points.  Including $(-2,3)$ eliminates $8$; including $(-1,3)$ and leaving out $(-2,3)$ eliminates $8$; including $(1,4)$ eliminates $9$; and including $(1,3)$ and leaving out $(1,4)$ eliminates $9$.  In all these cases, we can conclude that $P$ has at most $13$ lattice points in total.
    
    The only remaining case is that all lattice points of $P$ have heights between $-2$ and $2$.  The polygon can have at most one lattice point at height $-2$, at most one lattice point at height $2$, and some assortment of the $11$ total points with heights between $-1$ and $1$.  Once again, $P$ can have at most $13$ lattice points.

    \item  Suppose $k=5$.  If $P$ includes the point $(-4,3)$, then it cannot include $(-2,3)$, $(-1,2)$, or $(0,1)$.  Combined with the fact that $P$ can only have one lattice point at height $-2$, this leaves $P$ with at most $13$ total lattice points.  A similar argument holds if $P$ includes the point $(-2,3)$.  If $P$ contains neither $(-4,3)$ nor $(-2,3)$, then it has at most $1$ point at height $3$, at most one point at height $-2$, and some collection of the $11$ points between.  Thus $P$ has at most $13$ lattice points.
    \item  Suppose $k=6$.   Since $P$ has at most one lattice point at height $-2$, and only $12$ points are allowed outside of that height, $P$ has at most $13$ lattice points total.
    \item Suppose $k=7$.  Since $P$ has at most one lattice point at height $-2$, and only $11$ points are allowed outside of that height, $P$ has at most $12$ lattice points total.
\end{itemize}
We conclude that $|P\cap\mathbb{Z}^2|\leq 13$.
\end{proof}

As detailed in Appendix \ref{section:appendix}, we enumerated all non-hyperelliptic polygons containing the five prescribed points from the previous proof, along with some subset of the other $25$ permissible points. The end result was $69$ non-hyperelliptic panoptigons of lattice diameter $3$ or more, up to equivalence.  In the same appendix we show that there are $3$ non-hyperelliptic panoptigons with lattice diameter at most $2$, yielding a grand total of $72$ non-hyperelliptic panoptigons.  If we instead wish to count panoptigons of lattice width at least $3$, this count becomes $73$ due to the inclusion of $T_3$.

We remark that it is possible to give a much shorter proof that there are only finitely many non-hyperelliptic panoptigons. Suppose that $P$ is a panoptigon of lattice diameter $\ell(P)\geq 7$.  By the same argument that started our previous proof, we may assume without loss of generality that $P$ has $(0,0)$ as a panoptigon point as well as eight or more lattice points at height $-1$.  If $P$ contains a point of the form $(x,y)$ where $y\geq 2$, then the line segment $P\cap L$ where $L$ is the $x$-axis must have length at least $7\left(1-\frac{1}{y+1}\right)\geq 7\left(1-\frac{1}{2+1}\right)=\frac{14}{3}>4$.  As such $P$ must contain at least $4$ points at height $0$, impossible since there are only $3$ visible points at this height.  Similarly $P$ can have no lattice points at height $1$:  these would force the inclusion of either $(2,0)$ or $(-2,0)$.  Finally, if $P$ contains a  point of the form $(x,y)$ where $y\leq -3$, then the line segment $P\cap L'$ where $L'$ is the horizontal line at height $-2$ must have width at least $7\left(1-\frac{1}{|y|-1}\right)\geq 7\left(1-\frac{1}{3-1}\right)=\frac{7}{2}>3$.  As such we know that $P$ must contain at least $3$ lattice points at height $-2$, impossible since no two consecutive points at that height are both visible.  Thus we know that $P$ only has lattice points at heights $0$, $-1$, and $-2$, and so is a hyperelliptic polygon.  This means that if $P$ is a non-hyperelliptic panoptigon, it must have $\ell(P)\leq 6$.
 Since $|P\cap\mathbb{Z}^2|\leq (\ell(P)+1)^2$, it follows that if $P$ is a non-hyperelliptic panoptigon then it must have at most $(6+1)^2=49$ lattice points; there are any finitely many such polygons.  In principle one could enumerate all such polygons with at most $49$ lattice points as in \cite{Castryck2012} and check which are panoptigons; this would be much less efficient than the computation led to by our longer proof.

\section{Characterizing all maximal polygons of lattice width $3$ or $4$}\label{section:lw3_and_4}

In this section we will characterize all maximal polygons of lattice width $3$ or $4$.  By Lemma \ref{lemma:lw_facts}, this will allow us to determine which polygons of lattice width $1$ or $2$ can be the interior polygon of some lattice polygon. This will be helpful in Section \ref{section:big_face_graphs}, when we will need to know which of the infinitely many panoptigons of lattice width at most $2$ can be an interior polygon.

For lattice width $3$, we do have the triangle $T_3$ as an exceptional case; all other polygons with lattice width $3$ must have an interior polygon of lattice width $1$.

\begin{prop}\label{prop:lattice_width_3_maximal}
Let $P$ be a maximal polygon.
Then $P$ has lattice width $3$ if and only if up to equivalence we either have $P=T_3$, or  $P=T_{a,b}^{(-1)}$ where $a\geq \frac{1}{2}b-1$, $0\leq a\leq b$, and $b\geq 1$, and where $T_{a,b}\neq T_1$.
\end{prop}

\begin{proof}
If $P$ is equivalent to $T_3$, then it has lattice width $3$ as desired.  If $P$ is equivalent to some other $T_d$, then $P$ has lattice width $d\neq 3$, and so need not be considered.

Now assume $P$ is not equivalent to $T_d$ for any $d$, so that $P$ has lattice width $3$ if and only if $\pint$ has lattice width $1$ by Lemma \ref{lemma:lw_facts}.  This is the case if and only if $\pint$ is equivalent to $T_{a,b}$ for some $a,b\in\mathbb{Z}$ where $0\leq a\leq b$ and $b\geq 1$ (where $T_{a,b}\neq T_1$) by Theorem \ref{theorem:lw_012}.  Thus to prove our claim, it suffices by Proposition \ref{prop:interior_maximal} to show that $T_{a,b}^{(-1)}$ is a lattice polygon if and only if $a\geq \frac{1}{2}b-1$.

We set the following notation to describe $T_{a,b}$.  Starting with the face connecting $(0,0)$ and $(0,1)$ and moving counterclockwise, label the faces of $T_{a,b}$ as $\tau_1$, $\tau_2$, $\tau_3$, and $\tau_4$ (where $\tau_4$ does not appear if $a=0$).

Pushing out the faces, we find that $\tau_1^{(-1)}$ lies on the line $x=-1$, $\tau_2^{(-1)}$ on the line $y=-1$, $\tau_3^{(-1)}$ on the line $x+(b-a)y= b+1$, and $\tau_4^{(-1)}$ on the line $y=2$.  Note that working cyclically, we have $\tau_i^{(-1)}\cap\tau_{i+1}^{(-1)}$ is a lattice point:  we get the points $(-1,-1)$, $(2b-a+1,1)$, $(2a-b+1,2)$, and $(-1,2)$.  Thus if these are the vertices of $T_{a,b}^{(-1)}$, then $T_{a,b}^{(-1)}$ is a lattice polygon.
Certainly $(-1,-1)$ and $(2b-a+1,1)$ appear in $T_{a,b}^{(-1)}$.  The points $(2a-b+1,2)$ and $(-1,2)$ will appear as (not necessarily distinct) vertices of $T_{a,b}^{(-1)}$ if and only if  $2a-b+1\geq -1$; that is, if and only if $a\geq\frac{1}{2}b-1$.  Thus in the case that $a\geq\frac{1}{2}b-1$, we have that $T_{a,b}^{(-1)}$ is a lattice polygon with vertices at $(-1,-1)$, $(2b-a+1,-1)$, $(2a-b+1,2)$, and $(-1,2)$.

If on the other hand $a<\frac{1}{2}b-1$, then $\tau_4^{(-1)}$ is not a face of $T_{a,b}^{(-1)}$, and so one of the vertices of $T_{a,b}^{(-1)}$ is $\tau_1^{(-1)}\cap\tau_3^{(-1)}$.  These faces intersect at the point $\left(\frac{b+2}{b-a},-1\right)$, where we may divide by $b-a$ since $a<\frac{1}{2}b-1$ and so $a\neq b$. Note that $b-a> b-\frac{1}{2}b+1=\frac{1}{2}(b+2)$.  It follows that that $\frac{b+2}{b-a}<2$, and certainly  $\frac{b+2}{b-a}>1$, so $\left(\frac{b+2}{b-a},-1\right)$ is not a lattice point.  We conclude that $T_{a,b}^{(-1)}$ is a lattice polygon if and only if $a\geq \frac{1}{2}b-1$, thus completing our proof.
\end{proof}
The explicitness of this result, combined with the fact that $g\left(T_{a,b}^{(-1)}\right)=a+b+2$, allows us to count the number of maximal polygons $P$ of genus $g$ with lattice width $3$.  First, note that there are $\left\lfloor\frac{g-2}{2}\right\rfloor$ choices of $T_{a,b}$ with $g$ lattice points:  with our assumption that $a\leq b$, we can choose $a$ to be any number from $1$ up to $\left\lfloor\frac{g-2}{2}\right\rfloor$, and $b$ is determined from there.  Next, we will exclude those choices of $a$ that yield $a<\frac{1}{2}b-1$, or equivalently $a\leq \frac{1}{2}b-\frac{3}{2}$ since $a,b\in\mathbb{Z}$.  Given that $a+b=g$, this is equivalent to $a\leq\frac{1}{2}(g-a)-\frac{3}{2}$, or $\frac{3}{2}a\leq\frac{1}{2}g-\frac{3}{2}$, or $a\leq \frac{g}{3}-1$.  Thus the number of polygons we must exclude from the total count $\left\lfloor\frac{g-2}{2}\right\rfloor$ is $\left\lfloor \frac{g}{3}\right\rfloor-1$. 
We conclude that the number of maximal polygons of genus $g$ with lattice width $3$ is
\[\left\lfloor\frac{g-2}{2}\right\rfloor-\left\lfloor\frac{g}{3}\right\rfloor+1\]
when $g\geq 4$ (which allows us to ignore $T_3$).

We now wish to classify maximal polygons $P$ of lattice width $4$.  One possibility is that $P$ is $T_4$.  Other than this example, the interior polygon $P_{\textrm{int}}$ must have lattice width $2$.  Note that if $g(P_{\textrm{int}})=0$, then $P_{\textrm{int}}=T_2$; this has relaxed polygon $T_5$, which has lattice width $5$ and so is not under consideration.  If  $g(P_{\textrm{int}})=1$, then $\pint$ is one of the polygons in Figure \ref{figure:g1_lw2}.  It turns out that all of these can be relaxed to a lattice polygon, each of which has lattice width $4$; these polygons are illustrated in Figure \ref{figure:g1_relaxed}.

\begin{figure}[hbt]
\centering
\includegraphics[scale=0.8]{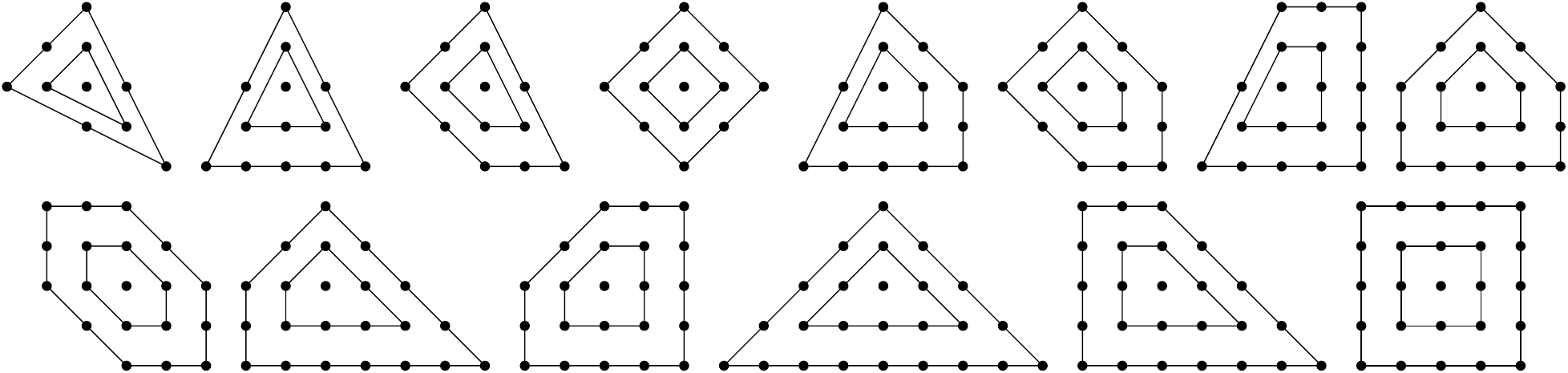}
\caption{The lattice width $4$ polygons with exactly one doubly interior point}
\label{figure:g1_relaxed}
\end{figure}

Now we deal with the most general case of polygons with $\textrm{lw}(P)=4$, namely those where $P_{\textrm{int}}$ has lattice width $2$ and genus $g'\geq 2$.  Thus $\pint$ must be one of the $\frac{1}{6}(g+3)(2g^2+15g+16)$ hyperelliptic polygons presented in Theorem \ref{theorem:lw_012}.  We must now determine which of these hyperelliptic polygons $Q$ have a relaxed polygon $Q^{(-1)}$ that has lattice points for vertices.  We do this over three lemmas, which consider the polygons of Type 1, Type 2, and Type 3 separately. 

\begin{lemma}\label{lemma:type1}
If $Q$ is of Type 1, then the relaxed polygon $Q^{(-1)}$ is a lattice polygon if and only if $i\leq \frac{3g+1}{2}$. 
\end{lemma}

\begin{proof}
Let $\tau_1$, $\tau_2$, $\tau_3$, and $\tau_4$ denote the four one-dimensional faces of $Q$, proceeding counterclockwise starting from the face connecting $(0,0)$ and $(1,2)$ (note that $\tau_4$ does not appear as a one-dimensional face if $i=2g$). Consider the relaxed faces $\tau_1^{(-1)}$, $\tau_2^{(-1)}$, $\tau_3^{(-1)}$, and $\tau_4^{(-1)}$. These lie on the lines $-2x+y=1$, $y=-1$, $2x+(2i-2g-1)y=2i+1$, and $y=3$.  Proceeding cyclically, the intersection points $\tau_i^{(-1)}\cap \tau_{i+1}^{(-1)}$ of these relaxed faces are $(-1,-1)$, $(2i-g,-1)$, $(3g-2i+2,3)$, and $(1,3)$. All these points are lattice points, so if they are indeed the vertices of $\pint$ then $Q^{(-1)}$ is a lattice polygon.

The one situation in which our relaxed polygon will not have all lattice points is if $\tau_1^{(-1)}$ and $\tau_3^{(-1)}$ intersect at a height strictly below $3$, cutting off the face $\tau_4^{(-1)}$ and yielding a vertex with $y$ coordinate strictly between $2$ and $3$.  These faces intersect at  $\left(\frac{g+1}{2(i-g)},\frac{i+1}{i-g}\right)$, which has $y$-coordinate strictly smaller than $3$ if and only if $\frac{i+1}{i-g}<3$, which can be rewritten as $i+1<3i-3g$, or as $\frac{3g+1}{2}<i$.  Thus when $i\leq \frac{3g+1}{2}$, our relaxed polygon is a lattice polygon; and when $i>\frac{3g+1}{2}$, it is not.
\end{proof}

\begin{lemma}\label{lemma:type2}
If $Q$ is of Type 2, then the relaxed polygon $Q^{(-1)}$ is a lattice polygon if and only if $i\geq \frac{g}{2}+1$ and $j\geq\frac{g-1}{2}$.
\end{lemma}

\begin{proof}
Label the faces of $Q$ cyclically as $\tau_1$, $\tau_2$, $\tau_3$, $\tau_4$, and $\tau_5$.  Due to the form of the slopes of these faces, the relaxed face $\tau_i^{(-1)}$ will intersect the relaxed face $\tau_{i+1}^{(-1)}$ at a lattice point; this is true for $\tau_1$ with $\tau_2$ and $\tau_5$ by computation, and for any horizontal line with a face of slope $1/k$ for some integer $k$.  Similarly, we are fine with the intersections of $\tau_3^{(-1)}$ and $\tau_4^{(-1)}$: these will always intersect at the lattice point $(g+2,1)$.  Thus the only way the relaxed polygon will fail to have lattice vertices is if certain edges are lost while pushing out.  Considering the normal fan of $Q$, this leads to two possible cases for $Q$ to not be integral:  if the face $\tau_2^{(-1)}$ is lost, and if the face $\tau_5^{(-1)}$ is lost.

First we consider the case that $\tau_2^{(-1)}$ is lost due to $\tau_1^{(-1)}$ and $\tau_3^{(-1)}$ intersecting at a point with $y$-coordinate strictly between $0$ and $-1$; note that this can only happen when $i<g$.  The face $\tau_1^{(-1)}$ is on the line $-2x+y=1$, and $\tau_3^{(-1)}$ is on the line $x-(g+1-i)y=i+1$.  These intersect at $\left(-\frac{g+2}{2g-2i+1},-\frac{2i+3}{2g-2i+1}\right)$. Note that $-\frac{2i+3}{2g-2i+1}>-1$ is equivalent to $\frac{2i+3}{2g-2i+1}<1$, which in turn is equivalent to $2i+3<2g-2i+1$.  This simplifies to $i<\frac{g}{2}+1$.  Thus we have a collapse of $\tau_2^{(-1)}$ that introduces a non-lattice vertex point if and only if $i<\frac{g}{2}+1$.

Now we consider the case that $\tau_5^{(-1)}$ is lost due to $\tau_1^{(-1)}$ and $\tau_4^{(-1)}$ intersecting at a point with $y$-coordinate strictly between $2$ and $3$.  The face $\tau_4^{(-1)}$ lies on the line with equation $x+(g-j)y=2g-j+2$.  This intersects $\tau_1^{(-1)}$ at $\left(\frac{g+2}{2g-2j+1},\frac{4g-2j+5}{2g-2j+1}\right)$.  Having $\frac{4g-2j+5}{2g-2j+1}<3$ is equivalent to $4g-2j+5< 6g-6j+3$, which can be rewritten as $4j<2g-2$, or $j< \frac{g-1}{2}$.  Thus we have a collapse of $\tau_5^{(-1)}$ that introduces a non-lattice vertex point if and only if $j< \frac{g-1}{2}$.

We conclude that $Q^{(-1)}$ is a lattice polygon if and only if $i\geq \frac{g}{2}+1$ and $j\geq\frac{g-1}{2}$
\end{proof}

\begin{lemma}\label{lemma:type3}
If $Q$ is of Type 3, then the relaxed polygon $Q^{(-1)}$ is a lattice polygon if and only if $i\geq g/2$ and $j\geq g/2$.
\end{lemma}

\begin{proof}Label the faces of $Q$ cyclically as $\tau_1,\ldots,\tau_6$, where $\tau_1$ is the face containing the lattice points $(k,2)$ and $(0,1)$ (with the understanding that some faces might not appear if one or more of $i$, $j$ and $k$ are equal to $0$).  If the faces $\tau_1^{(-1)},\ldots,\tau_6^{(-1)}$ are all present in the polygon $P^{(-1)}$, then they intersect at lattice points by the arguments from the previous proof.  Thus we need only be concerned with the following cases:  where $\tau_3^{(-1)}$ collapses due to $\tau_2^{(-1)}$ and $\tau_4^{(-1)}$ intersecting at a point $(x,y)$ with $0>y>-1$; and where $\tau_6^{(-1)}$ collapses due to $\tau_5^{(-1)}$ and $\tau_1^{(-1)}$ intersecting at a point $(x,y)$ with $2<y<3$.

First we consider $\tau_2^{(-1)}$ and $\tau_4^{(-1)}$.  We have that $\tau_2^{(-1)}$ lies on the line defined by $x=-1$, and that $\tau_4^{(-1)}$ lies on the line defined by $x-(g+1-i)y=i+1$.  These lines intersect at $(-1,-\frac{i+2}{g+1-i})$.  The $y$-coordinate is strictly greater than $-1$ when $\frac{i+2}{g+1-i}<1$, i.e. when $i+1<g+1-i$, which can be rewritten as $i<\frac{g}{2}$.  Thus we lose $\tau_3^{(-1)}$ to a non-lattice vertex  precisely when $i<\frac{g}{2}$.

Now we consider $\tau_5^{(-1)}$ and $\tau_1^{(-1)}$.  We have that  $\tau_1^{(-1)}$ lies on the line $x-ky=-k+1$, unless $k=0$ in which case it lies on the line $x=-1$; and that $\tau_5^{(-1)}$ lies on the line $x+(g+1-k-j)y=2g+2-k-j$.  In the event that $k\neq 0$, these intersect at $\left(\frac{gk+g-j+1}{g-j+1},\frac{2g-j+1}{g-j+1}\right)$, which has $y$-coordinate strictly smaller than $3$ when $\frac{2g-j+1}{g-j+1}<3$, or equivalently if $2g-j+1<3g-3j+1$, or equivalently if $j<\frac{g}{2}$.  For the $k=0$ case, the intersection point becomes $\left(-1,\frac{2g-j+3}{g-j+1}\right)$, which has $y$-coordinate strictly smaller than $3$ when $\frac{2g-j+3}{g-j+1}<3$, or equivalently when $2g-j+3<3g-3j-3k+3$, or equivalently when $j<\frac{g}{2}$.  Thus we have a non-lattice vertex due to $\tau_5^{(-1)}$ collapsing precisely when $j<\frac{g}{2}$.

We conclude that $Q^{(-1)}$ is a lattice polygon if and only if $i\geq g/2$ and $j\geq g/2$.
\end{proof}

Combining Lemmas \ref{lemma:type1}, \ref{lemma:type2}, and \ref{lemma:type3} and the preceding discussion, we have the following classification of maximal polygons with lattice width $4$.

\begin{prop}\label{prop:lattice_width_4_maximal}

Let $P$ be a maximal polygon of lattice width $4$.  Then up to lattice equivalence, $P$ is either $T_4$; one of the $14$ polygons in Figure \ref{figure:g1_relaxed}; or $Q^{(-1)}$, where $Q$ is a hyperelliptic polygon satisfying the conditions of Lemma \ref{lemma:type1}, \ref{lemma:type2}, or \ref{lemma:type3}.
\end{prop}

The most important consequence of Propositions \ref{prop:lattice_width_3_maximal} and \ref{prop:lattice_width_4_maximal} is that we can determine which panoptigons of lattice width $1$ or lattice width $2$ are interior polygons of some lattice polygon.  We summarize this with the following result.

\begin{cor}\label{corollary:lw12_panoptigon_point_bound}
Let $Q$ be a panoptigon with $\lw(Q)\leq 2$ such that $Q^{(-1)}$ is lattice polygon.  Then $|Q\cap\mathbb{Z}^2|\leq 11$.
\end{cor}

\begin{proof}
If $\lw(Q)= 1$ with $Q^{(-1)}$ a lattice polygon, then $Q$ must be the trapezoid $T_{a,b}$ with $0\leq a\leq b$, $b\geq 1$, and $a\geq \frac{b}{2}-1$ by Proposition \ref{prop:lattice_width_3_maximal}.  In order for $T_{a,b}$ to be a panoptigon, we need $a\leq 2$ by Lemma \ref{lemma:panoptigon_g0}, so $2\geq \frac{b}{2}-1$, implying $b\leq 6$.  It follows that $|Q\cap\mathbb{Z}^2|=a+b+2\le 2+6+2=10$.

Now assume $\lw(Q)= 2$  with $Q^{(-1)}$ a lattice polygon.  If $Q$ has genus $0$ then it is $T_2$, and has $6$ lattice points.  If $Q$ has genus $1$ then it is one of the polygons in Figure \ref{figure:g1_lw2}, and so has at most $9$ lattice points.  Outside of these situations, we know that $Q$ is a  hyperelliptic panoptigon of genus $g\geq 2$ as characterized in Lemma \ref{lemma:hyperelliptic_panoptigon}. We deal with two cases:  where $Q$ has a panoptigon point at height $1$, and where it does not.

In the first case, we either have $g=2$ with $Q$ of Type 1 or Type 2, or $g=3$ with $Q$ of Type 1.  A hyperelliptic polygon of Type 1 has $(i+1)+(1+2g-i)=2g+2$ boundary points. A hyperelliptic polygon of Type 2 has $i+j+3$ boundary points.  If $Q$ is of Type 1, then it has in total $3g+2\leq11 $ lattice points.  If $Q$ is of Type 2, then $i+j\leq 2g+1=2\cdot 2+1=5$, implying that $Q$ has a total of $i+j+3+g\leq 5+3+2=10$ lattice points. 

In the second case, we know that $Q$ must have at most $3$ points at height $0$ or $2$, and exactly $1$ point at the other height. First we claim that $Q$ cannot be of Type 1: there are $2g+2\geq 6$ boundary points, all at height $0$ or $2$, and $Q$ can have at most $4$ points total at those heights.  For Types 2 and 3,  we know by Lemmas \ref{lemma:type2} and \ref{lemma:type3} that either $i\geq\frac{g}{2}+1$ and $j\geq \frac{g-1}{2}$, or  $i\geq\frac{g}{2}$ and $j\geq \frac{g}{2}$.  At least one of $i$ and $j$  must equal $0$ to allow for a single point at height $0$ or height $2$, so these inequalities are impossible for $g\geq 2$.  Thus $Q$ cannot have Type 2 or Type 3 either, and this case never occurs.

We conclude that if $Q$ is a panoptigon of lattice width $1$ or $2$ such that $Q^{(-1)}$ is a lattice polygon, then $|Q\cap\mathbb{Z}^2|\leq 11$.

\end{proof}

\section{Big face graphs are not tropically planar}\label{section:big_face_graphs}

Let $G$ be a planar graph.  Recall that we say that $G$ is a \emph{big face graph} if for any planar embedding of $G$,  there exists a bounded face that shares an edge with every other bounded face.  Our main examples of big face graphs will come from the following construction.  First we recall the construction of a \emph{chain} of genus $g$ from \cite[\S 6]{BJMS}.  Start with with $g$ cycles in a row, connected at $g-1$ $4$-valent vertices.  We will resolve each of these $4$-valent vertices to result in two $3$-valent vertices in one of two ways.  Let $v$ be a vertex, incident to the edges $e_1,e_2,f_1,f_2$ where $e_1$ and $e_2$ are part of one cycle and $f_1$ and $f_2$ are part of another.  We will remove $v$ and replace it with two connected vertices $v_1$ and $v_2$, and we will either connect $v_1$ to $e_1$ and $f_1$ and $v_2$ to $e_2$ and $f_2$; or we will connect $v_1$ to $e_1$ and $e_2$ and $v_2$ to $f_1$ and $f_2$.  Any graph obtained from making such a choice at each vertex is then called a chain.  Figure \ref{figure:chain_construction} illustrates, for $g=3$, the starting $4$-regular graph; the two ways to resolve a $4$-valent vertex; and the resulting chains of genus $3$.  We remark that although there are $2\times 2 = 4$ ways to choose the vertex resolutions, two of them yield isomorphic graphs, giving us $3$ chains of genus $3$ up to isomorphism.  Note that for every genus, there is exactly one chain that is bridge-less, i.e. $2$-edge-connected.

\begin{figure}[hbt]
\centering
\includegraphics{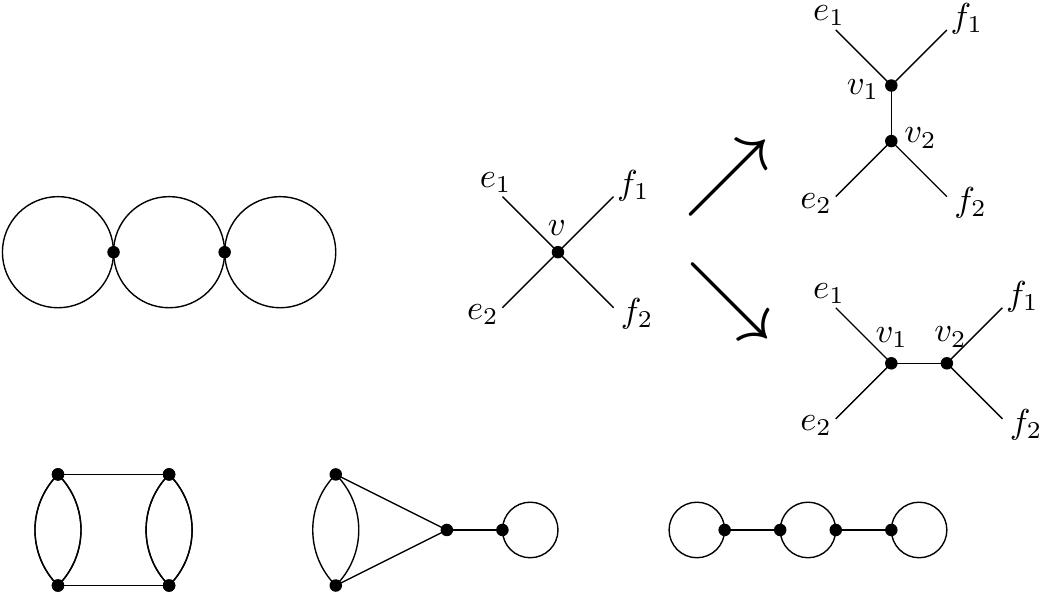}
\caption{The starting $4$-regular graph in the chain construction; the two choices for resolving a $4$-valent vertex; and the three chains of genus $3$, up to isomorphism}
\label{figure:chain_construction}
\end{figure}

 Given a chain of genus $g$, we construct a \emph{looped chain} of genus $g+1$ by adding an edge from the first cycle to the last one.  The looped chains of genus $4$ corresponding to the chains of genus $3$ are illustrated in Figure \ref{figure:chains_and_looped_chains}. For larger genus, we remark that two non-isomorphic chains can give rise to isomorphic looped chains.

\begin{figure}[hbt]
\centering
\includegraphics{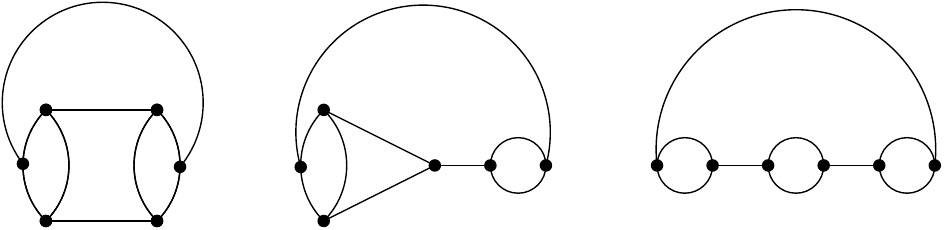}
\caption{The looped chains of genus $4$ }
\label{figure:chains_and_looped_chains}
\end{figure}

In order to argue that any looped chain is a big face graph, we recall the following useful result.   By a special case of Whitney's $2$-switching theorem \cite[Theorem 2.6.8]{graphs_on_surfaces}, if $G$ is a $2$-connected graph, then any other planar embedding can be reached, up to weak equivalence\footnote{Weak equivalence means two graph embeddings have the same facial structure, although possibly with different unbounded faces.}, from the standard embedding by a sequence of \emph{flippings}.  A flipping of a planar embedding finds a cycle $C$ with only two vertices $v$ and $w$ incident to edges exterior to $C$, and then reverses the orientation of $C$ and all vertices and edges interior to $C$ to obtain a new embedding.  This process is illustrated in Figure \ref{figure:2-flip}, where $C$ is the highlighted cycle $v-a-b-w-d-v$.

\begin{figure}[hbt]
\centering
\includegraphics{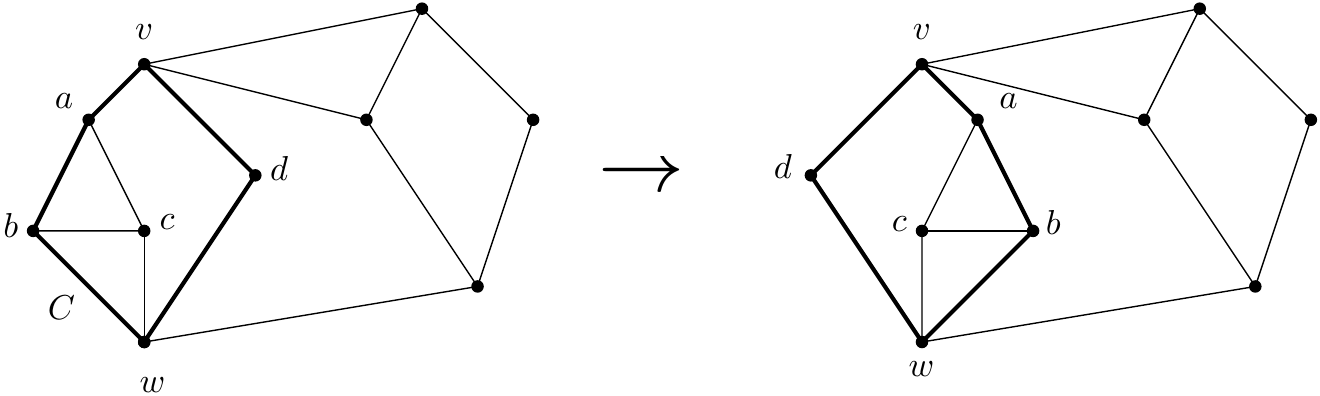}
\caption{Two embeddings of a planar graph related by a flipping}
\label{figure:2-flip}
\end{figure}

\begin{lemma}\label{lemma:looped_chain}
Any looped chain is a big face graph.
\end{lemma}

\begin{proof}   In the standard embedding of a looped chain as in Figure \ref{figure:chains_and_looped_chains}, there are (at least) two faces that share an edge with all other faces:  one bounded and one unbounded.  Since any looped chain is $2$-connected, any other embedding can be reached, up to weak equivalence, by a sequence of flippings.   It thus suffices to show that the standard embedding of a looped chain is invariant under flipping.

Consider the standard embedding of a looped chain $G$, and assume that that $C$ is a cycle in $G$ that has exactly two vertices $v$ and $w$ incident to edges exterior to $C$.  Let $\overline{C}$ denote the set of all vertices in or interior to $C$.  Since $G$ is trivalent and $C$ is $2$-regular, we know that $v$ and $w$ are each incident to exactly one edge, say $e$ for $v$ and $f$ for $w$, that is exterior to $C$.  We now deal with two possibilities:  that $\overline{C}=V(G)$, and that $\overline{C}\subsetneq V(G)$.

If $\overline{C}=V(G)$, then $e=f$, and the only possibility is that $v$ and $w$ are the vertices added to a chain $H$ to build the looped chain $G$; that $H$ is the bridge-less chain; and that $C$ is the outside boundary of $H$ in its standard embedding.  Flipping with respect to $C$ does not change the embedding of this graph.

\begin{figure}[hbt]
\centering
\includegraphics{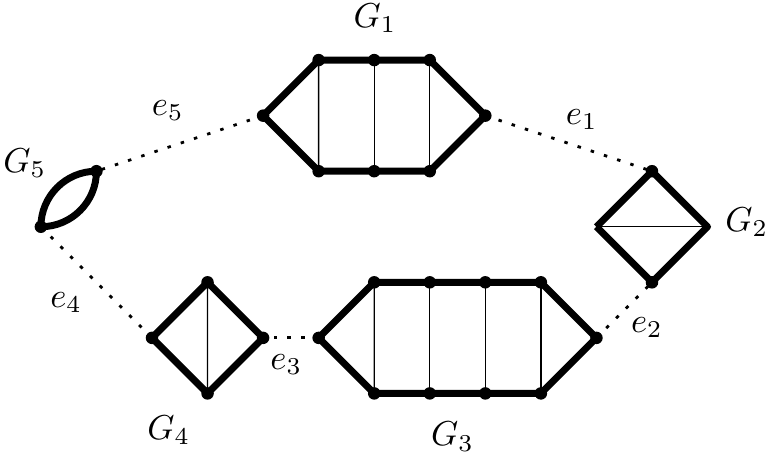}
\caption{The structure of a looped chain, where the bridge-less chains $G_i$ have solid edges and the edges $e_i$ are dotted; the boundaries of the $G_i$ are bold, and are the only possible choices of $C$ for a flipping}
\label{figure:looped_chain_possible_cycles}

\end{figure}

If $\overline{C}\subsetneq V(G)$, then $\{e,f\}$ forms a $2$-edge-cut for $G$, separating it into $\overline{C}$ and $\overline{C}^C$.  Consider the structure of $G$:  is is a collection of $2$-edge-connected graphs $G_1,\ldots,G_k$, namely a collection of bridge-less chains, connected in a loop by edges $e_1,\ldots,e_k$, where $e_i$ connects $G_i$ and $G_{i+1}$, working modulo $k$; see Figure \ref{figure:looped_chain_possible_cycles} for this labelling scheme.  We claim that $e,f\in\{e_1,\ldots,e_k\}$.  If not, then without loss of generality $e$ is in some bridge-less chain $G_i$.  If $f\in E(G_j)$ for $j\neq i$, then the graph remains connected; the same is true if $f\in\{e_1,\ldots,e_k\}$.  So we would need $f$ to also be in $G_i$.  By the structure of the looped chain, we would the removal of $e$ and $f$ to disconnect $G_i$ into multiple components, at least one of which is not incident to $e_i$ or $e_{i+1}$; however, this is impossible based on the structure of a bridge-less chain. It follows that $e$ and $f$ must be among $e_1,\ldots,e_k$.  The only way to choose a pair $\{e,f\}$ from among $e_1,\ldots,e_k$ so that they are the only exterior edges incident to the boundary of a cycle $C$ is if they are incident to the same bridge-less chain $G_i$; that is, if up to relabelling we have  $e=e_i$ and $f=e_{i+1}$ for some $i$.  Thus $C$ and its interior constitutes one of the bridge-less chains $G_i$.  But flipping a bridge-less chain does not change the embedding of our (unlabelled) graph, completing the proof.
\end{proof}

\begin{figure}[hbt]
\centering
\includegraphics{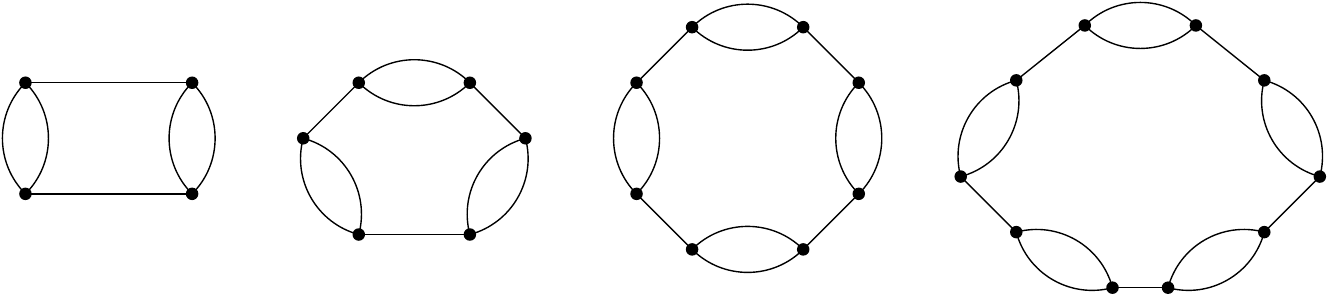}
\caption{The loop of loops $L_g$ for $3\leq g\leq 6$}
\label{figure:loop_of_loops}

\end{figure}

We summarize the connection between big face graphs and panoptigons in the following lemma.

\begin{lemma}\label{lemma:big_face_panoptigon_lemma}
Suppose that $G$ is a tropically planar big face graph arising from a polygon $P$.  Then $\pint$ is a panoptigon.
\end{lemma}

This is not an if-and-only-if statement, since not all triangulations of $P$ connect a point of $\pint$ to all other points of $\pint$; for instance, the chain of genus $3$ with two bridges is not a big face graph, but by \cite[\S 5]{BJMS} it arises from $T_4$ whose interior polygon is a panoptigon.

\begin{proof}
Let $\Delta$ be a regular unimodular triangulation of $P$ such that $G$ is the skeleton of the weak dual graph of $\Delta$.  The embedding of $G$ arising from this construction must have a bounded face $F$ bordering all other faces. By duality, we know that $F$ corresponds to an interior lattice point $p$ of $P$.  Since $F$ shares an edge with all other bounded faces, dually $p$ is connected to each other interior point of $P$ by a primitive edge in $\Delta$.  Thus $\pint$ is a panoptigon, with $p$ a panoptigon point for it.
\end{proof}

One common example of a looped chain of genus $g$ is the \emph{loop of loops} $L_g$,  obtained by connecting $g-1$ bi-edges in a loop. This is illustrated in Figure \ref{figure:loop_of_loops} for $g$ from $3$ to $6$.  For low genus, the loop of loops is tropically planar. Figure \ref{figure:triangulations_for_lol} illustrates polygons of genus $g$ for $3\leq g\leq 10$ along with collections of edges emanating from an interior point; when completed to a regular unimodular triangulation\footnote{One way to see that this can be accomplished is to use a placing triangulation \cite[\S 3.2.1]{triangulations}, where the highlighted panoptigon point is placed first and the other lattice points are placed in any order.}, they will yield $L_g$ as the dual tropical skeleton.  Thus $L_g$ is tropically planar for $g\leq 10$.  Another example of a tropically planar looped chain, this one of genus $11$, is pictured in Figure \ref{figure:g11_big_face}, along with a regular unimodular triangulation of a polygon giving rise to it.  Since the theta graph of genus $2$ is also tropically planar \cite[Example 2.5]{BJMS} and is a big face graph, there  exists at least one tropically planar big face graph of genus $g$ for $2\leq g\leq 11$.  We are now ready to prove that this does not hold for $g\geq 14$.

\begin{figure}[hbt]
\centering
\includegraphics[scale=0.8]{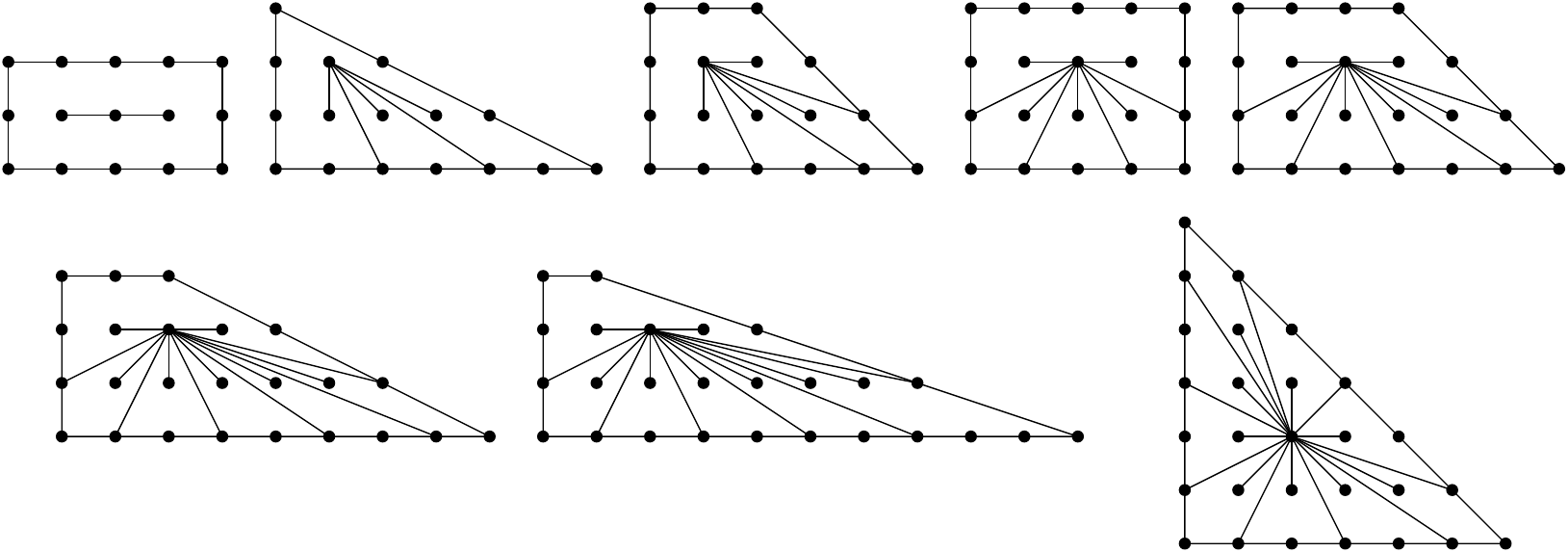}
\caption{Starts of triangulations that will yield the loop of loops as the dual tropical skeleton }
\label{figure:triangulations_for_lol}

\end{figure}

\begin{figure}[hbt]
\centering
\includegraphics{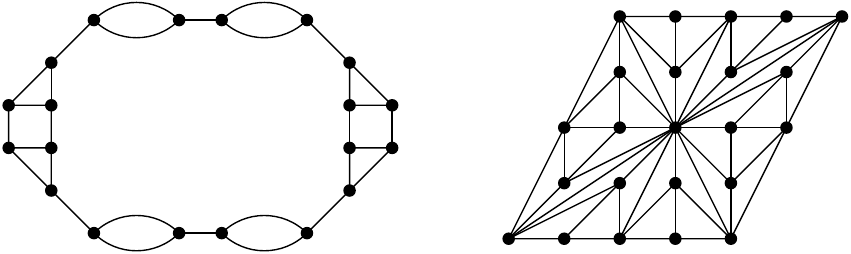}
\caption{A tropically planar big face graph of genus $11$, with a regular unimodular triangulation giving rise to it }
\label{figure:g11_big_face}

\end{figure}

\begin{proof}[Proof of Theorem \ref{theorem:big_face_graphs}]  Let $G$ be a tropically planar big face graph, and let $P$ be a lattice polygon giving rise to it. By Lemma \ref{lemma:big_face_panoptigon_lemma}, $\pint$ is a panoptigon.  If $\lw(\pint)\leq 2$, then $g=|\pint\cap\mathbb{Z}^2|\leq 11$ by Corollary \ref{corollary:lw12_panoptigon_point_bound}.  If $\lw(\pint)\geq 3$, then $g=|\pint\cap\mathbb{Z}^2|\leq 13$ by Theorem \ref{theorem:at_most_13}.  Either way, we may conclude that the genus of $G$ is at most $13$.
\end{proof}

It follows, for instance, that no looped chain of genus $g\geq 14$ is tropically planar.

If we are willing to rely on our computational enumeration of all non-hyperelliptic panoptigons, we can push this further:  there does not exist a tropically planar big face graph for $g\geq 12$, and this bound is sharp.  We have already seen in Figure \ref{figure:g11_big_face} that there exists a tropically planar big face graph of genus $11$.  To see that none have higher genus, first note that if $\pint$ is a panoptigon with $12$ or $13$ lattice points, then $\pint$ must be non-hyperelliptic by Corollary \ref{corollary:lw12_panoptigon_point_bound}.  Thus $\pint$ must be one of the $15$ non-hyperelliptic panoptigons with $12$ lattice points, or one of the $8$ non-hyperelliptic panoptigons with $13$ lattice points, as presented in Appendix \ref{section:appendix}.  However, for each of these polygons $Q$, we have verified computationally that $Q^{(-1)}$ is not a lattice polygon; see Figure \ref{figure:panoptigons_12_or_13}.  Thus no lattice polygon of genus $g\geq 12$ has an interior polygon that is also a panoptigon.  It follows from Lemma \ref{lemma:big_face_panoptigon_lemma} that no big face graph of genus larger than $11$ is tropically planar.

We close with several possible directions for future research.

\begin{itemize}
    \item For any lattice point $p$, let $\textrm{vis}(p)$ denote the set of all lattice points visible to $p$ (including $p$ itself).  Given a convex lattice polygon $P$, define its \emph{visibility number} to be the minimum number of lattice points in $P$ needed so that we can see every lattice point from one of them:
\[V(P)=\min\left\{|S|\,:\, S\subset P\cap\mathbb{Z}^2\textrm{ and } P\cap\mathbb{Z}^2\subset \bigcup_{p\in S}\textrm{vis}(p) \right\}.\]
Thus $P$ is a panoptigon if and only if $V(P)=1$.  Classifying polygons of fixed visibility number $V(P)$, or finding relationships between $V(P)$ and such properties as genus and lattice width, could be interesting in its own right, and could provide new criteria for determining whether graphs are tropically planar; for instance, the prism graph $P_n=K_2\times C_n$ can only arise from a polygon $P$ with $V(P)\leq 2$.  This question is in some sense a lattice point version of the art gallery problem. 

\item  We can generalize from two-dimensional panoptigons to $n$-dimensional \emph{panoptitopes}, which we define to be convex lattice polytopes containing a lattice point $p$ from which all the polytope's other lattice points are visible. A few of our results generalize immediately; for instance, the proof of Lemma \ref{lemma:panoptigon_g1} works in $n$-dimensions, so any polytope with exactly one interior lattice point is a panoptitope. A complete classification of $n$-dimensional panoptitopes for $n\geq 3$ will be more difficult than it was in two-dimensions, especially since it is no longer the case that there are finitely many polytopes with a fixed number of lattice points.  Results about panoptitopes would also have applications in tropical geometry; for instance, an understanding of three-dimensional panoptitopes would have implications for the structure of tropical surfaces in $\mathbb{R}^3$.

\item  To any lattice polygon we can associate a toric surface \cite{toric_varieties}.  An interesting question for future research would be to investigate those toric surfaces that are associated to panoptigons, or more generally toric varieties associated to panoptitopes.

\end{itemize}

\appendix

\begin{appendices}

\section{Panoptigon computations}
\label{section:appendix}

From the proof of Theorem \ref{theorem:at_most_13}, we know that any panoptigon of lattice width and lattice diameter both at least $3$ must be equivalent to a polygon consisting of some subset of the thirty lattice points pictured in Figure \ref{figure:ruling_out_points}, where the points $(0,0)$, $(-1,-1)$, $(0,-1)$, $(1,-1)$, and $(2,-1)$ must be included.  Using \texttt{polymake} \cite{polymake}, we ran through all possible convex polygons consisting only of these $30$ points.  Ruling out those without interior lattice points or with all lattice points collinear, we found a total of $215$ distinct polygons, some of which were equivalent under a unimodular transformation. These $215$ polygons are available as the collection \enquote{Non-hyperelliptic Panoptigons} in polyDB \cite{polydb:paper} at \url{https://db.polymake.org}.  Eliminating redundant copies, we find that there are a total of $69$ non-hyperelliptic panoptigons of lattice width and lattice diameter both at least $3$, up to lattice equivalence; these appear in Figure \ref{figure:panoptigons_all}. %Combined with the $3$ nonhyperelliptic panoptigons of lattice width $2$ from Proposition \ref{prop:possibilities_ld2}, this yields the count of $72$ nonhyperelliptic panoptigons  in Corollary \ref{cor:nonhyperelliptic_panoptigon_count}.
  The panoptigons with $12$ or $13$ lattice points appear in Figure \ref{figure:panoptigons_12_or_13}, along with their relaxed polygons.  Each relaxed polygon has at least one nonlattice vertex, marked by a square.  The computation of these relaxed polygons verifies that no non-hyperelliptic panoptigon with $12$ or $13$ lattice points is the interior polygon of a lattice polygon. 

To complete an enumeration of all non-hyperelliptic panoptigons of genus $g\geq 3$, it remains to find those panoptigons $P$ that have lattice diameter smaller than $3$.
We accomplish this with the following proposition.

\begin{prop}\label{prop:possibilities_ld2}
Let $P$ be a non-hyperelliptic panoptigon of lattice diameter at most $2$.  Then up to lattice equivalence $P$ is either the triangle $\textrm{conv}((0,1),(0,3),(4,0))$, the quadrilateral $\textrm{conv}((1,0),(2,0),(3,1),(0,3))$, or the quadrilateral $\textrm{conv}((0,1),(0,2),(2,3),(3,0))$.
\end{prop}

These three polygons are illustrated in Figure \ref{figure:possibilities_ld2}.

\begin{figure}[hbt]
\centering
\includegraphics{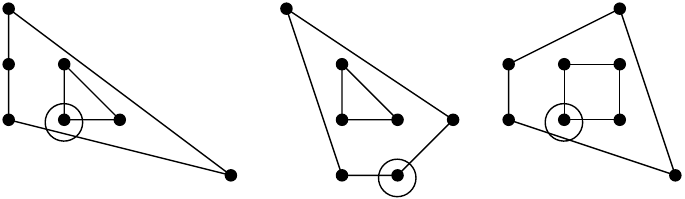}
\caption{The three non-hyperelliptic panoptigons from Proposition \ref{prop:possibilities_ld2}}
\label{figure:possibilities_ld2}

\end{figure}

\begin{proof}

Since $P$ is non-hyperelliptic, we know that $\lw(P)\geq 3$.  
Note that we cannot have $\ell(P)=1$, since then we would have $\lw(P)\leq\lfloor\frac{4}{3}\ell(P)\rfloor+1=2$.  Thus $\ell(P)=2$.  It follows that  $\lw(P)\leq\lfloor\frac{4}{3}\ell(P)\rfloor+1$=2+1=3, so $\lw(P)=3$.  We know $P$ is not $T_3$ since $T_3$ is hyperelliptic, so we know that the interior polygon $P_{\textrm{int}}$ must have lattice width $1$.  It follows that $\pint$ must be a trapezoid of height $1$, and since $\ell(P)=2$ that trapezoid must have at most $3$ lattice points at each height; thus $\pint=T_{a,b}$ where $0\leq a\leq b\leq 2$.  It follows that $P$ must be contained in one of the polygons pictured in Figure \ref{figure:possibilities_for_pint}; these are the maximal polygons associated to the candidates for $\pint$.  In order to refer to the lattice points of these polygons with coordinates, we will assume that each is positioned to have the lower left corner at the origin $(0,0)$.

\begin{figure}[hbt]
\centering
\includegraphics{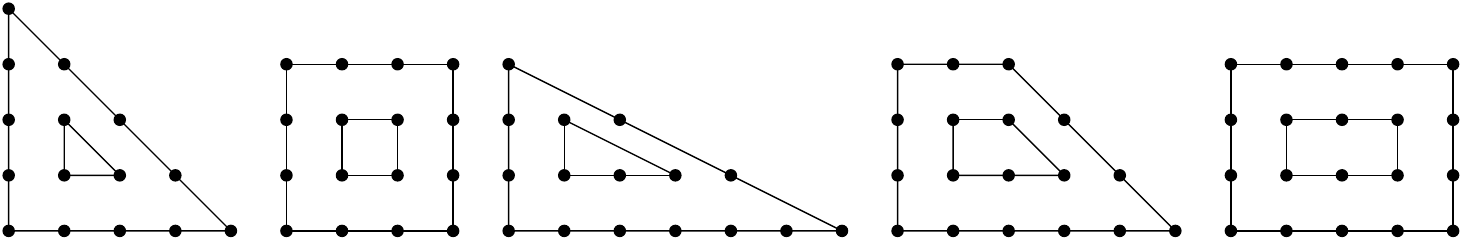}
\caption{Possible interior polygons for $\pint$, and polygons that must contain $P$}
\label{figure:possibilities_for_pint}

\end{figure}

We claim that $P$ cannot have $T_{0,2}$ $T_{1,2}$, or $T_{2,2}$ as its interior polygon. In each of those cases, note that $P$ automatically has  $3$ interior points at height $1$; since $\ell(P)=2$, there can be no boundary lattice points at height $1$.  In order for boundary lattice points at height $0$ to connect to boundary lattice points at height greater than $1$, the points $(1,0)$ and  $(4,0)$ must be included; but then there are $4$ collinear lattice points at height $0$, contradicting $\ell(P)=2$.

Now suppose $\pint=T_{1,1}$.  Note that no boundary point of the $3\times 3$ square can see all interior points, so any panoptigon point $q$ must be an interior point. Without loss of generality, assume that it is $q=(1,1)$, meaning that the points $(1,3)$, $(3,1)$, and $(3,3)$ cannot be included in $P$.  Among the two points $(2,3)$ and $(3,2)$, at least one must be included to allow for the desired interior polygon.  By symmetry we may assume that $(2,3)$ is included.  There cannot be any other points at height $3$, so $(2,3)$ must be a vertex of $P$ and connect to a boundary point of the form $(0,b)$; the only possible such point is $(0,2)$.  It then follows that $(3,2)$ cannot be included, since this would yield $4$ collinear points at height $2$.  Thus $P$ has an edge connecting $(2,3)$ to $(3,0)$.  The point $(2,0)$ cannot appear in $P$ since there are already three points with $x$-coordinate equal to $2$, so $(3,0)$ must be connected to $(0,1)$.  At this point, we know that $P=\textrm{conv}((0,1),(0,2),(2,3),(3,0))$.  This is indeed a panoptigon of lattice width $3$ and lattice diameter $2$.

Finally we will deal with the case where $\pint=T_{0,1}$.  We deal with several possibilities for the (not necessarily unique) panoptigon point $q$ of $P$.

\begin{itemize}
    \item Suppose $q$ is an interior lattice point of $P$.  By symmetry we may assume $q=(1,1)$, so the lattice points $(3,1)$ and $(1,3)$ cannot be included.  Since there must at least one lattice point at height $3$ or above, either $(0,3)$ or $(0,4)$ (or both) must be included.  If it is only $(0,4)$ and not $(0,3)$, then the points $(1,0)$ and $(3,0)$ must be included, yielding the polygon $\textrm{conv}((1,0),(3,0),(0,4))$.  Otherwise, $(0,3)$ is in $P$.  Since $(0,3)$, $(1,2)$, and $(2,1)$ are all lattice points of $P$, the point $(3,0)$ cannot be included.  Now, at least one lattice point from the diagonal edge must be included, namely $(4,0)$, $(2,2)$, or $(0,4)$; in fact, it must be exactly one, since otherwise $(1,3)$ or $(3,1)$ would be introduced by convexity.  If $P$ contains $(0,4)$ or $(2,2)$ and no other points along that edge, then it must also contain $(3,0)$, which we have already ruled out.  Thus $P$ contains $(4,0)$, and as it does not contain $(3,0)$ it must have an edge connecting $(4,0)$ to $(0,1)$.  At this point there is a single possibility for $P$, namely $P=\textrm{conv}((0,1),(0,3),(4,0))$.  This polygon is  equivalent to the previous one, so we need only include one.  This panoptigon does indeed have lattice diameter $2$.

    \item  Now we deal with the case that the panoptigon point is a boundary point. Since the panoptigon point must see all three interior points, it must either be a vertex of $T_4$ or the midpoint of one of the edges.  Up to symmetry, we may thus assume that $q$ is either $(0,0)$ or $(2,0)$.  If $q=(0,0)$, then the point $(1,3)$ must be included; otherwise we would need $(0,3)$, which is not visible to $(0,0)$.  Similarly $(1,3)$ is included, but then $(2,2)$ is included by convexity, and this point is not visible to $(0,0)$, a contradiction.
    
    If $q=(2,0)$, then there are $1$, $2$, or $3$ points at height $0$.  If there is only $q$, then the points $(0,1)$ and $(3,1)$ must be included, contradicting $\ell(P)=2$.  If there are $2$ points, we will assume by symmetry that the two points are $(1,0)$ and $(2,0)$.  The lattice point $(3,1)$ must then be included and the point $(0,1)$ must not be included; the only remaining point to include from the face on the line $x=0$ is $(0,3)$. No other lattice points can be included, so then $P=\textrm{conv}((1,0),(2,0),(3,1),(0,3))$.  Finally, if there are $3$ points at height $0$ they must be $(1,0)$, $(2,0)$, and $(3,0)$.  But now neither $(0,3)$ nor $(1,3)$ may be included since $\ell(P)=2$.  Since $(0,4)$ is not visible from $q$, there are no points in $P$ with height greater than $2$, a contradiction to $T_{0,1}$ being the interior polygon of $P$.
\end{itemize}
We conclude that the only non-hyperelliptic panoptigons $P$ with $\ell(P)\leq 2$ are the three claimed.

\end{proof}

Combined with our computation, this gives us the following count.

\begin{cor}\label{cor:nonhyperelliptic_panoptigon_count}
Up to lattice equivalence, there are $72$ non-hyperelliptic panoptigons.
\end{cor}

The explicitness of our enumeration allows us to find the largest lattice width of any panoptigon: by Lemma \ref{lemma:lw_facts}, the lower right triangle in Figure \ref{figure:panoptigons_12_or_13} has lattice width $5$, and all the other panoptigons have lattice width $4$ or less.

\begin{figure}[hbt]
\centering
\includegraphics[scale=0.75]{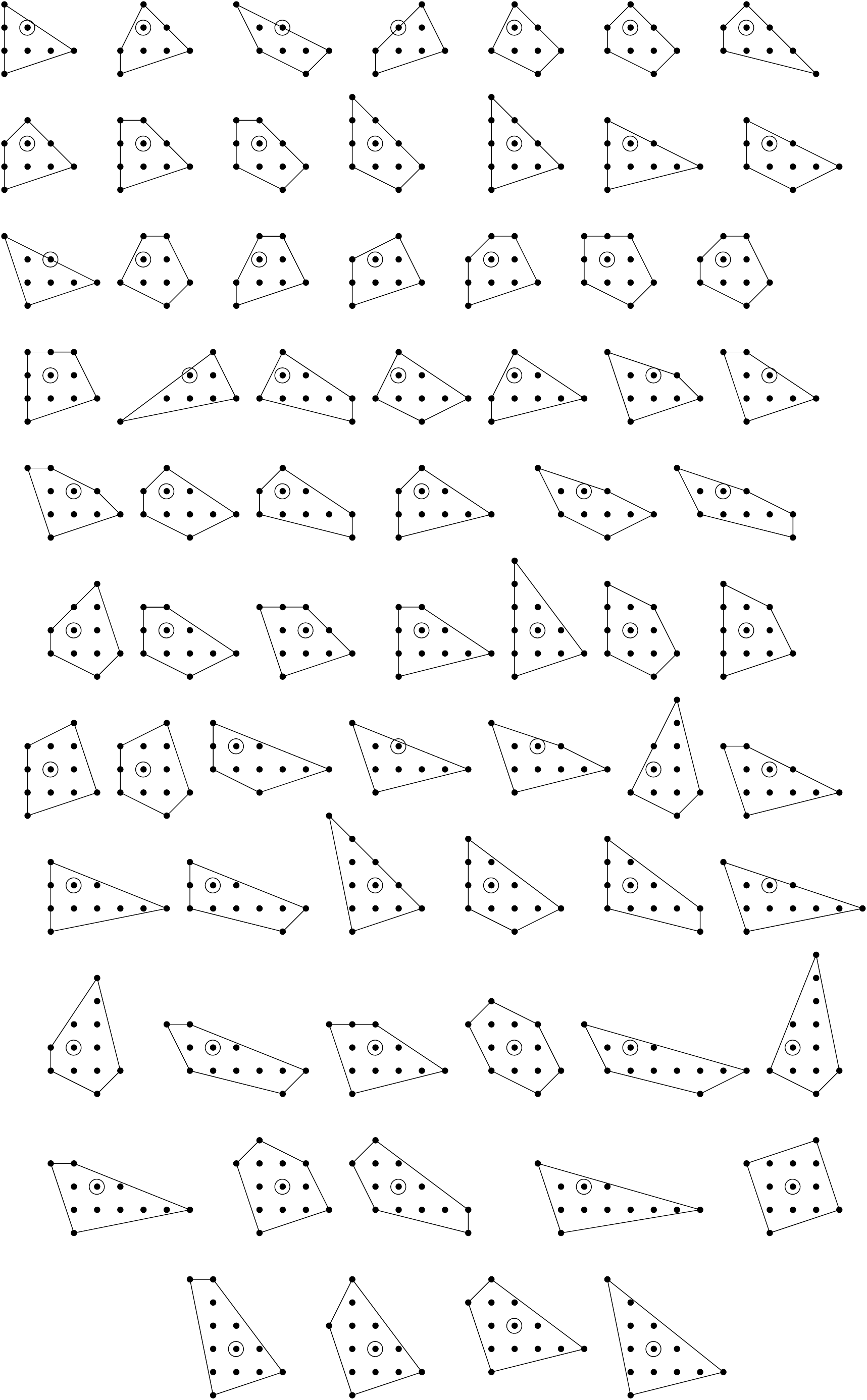}
\caption{All non-hyperelliptic panoptigons with lattice diameter at least $3$}
\label{figure:panoptigons_all}

\end{figure}

\begin{figure}[hbt]
\centering
\includegraphics[scale=0.75]{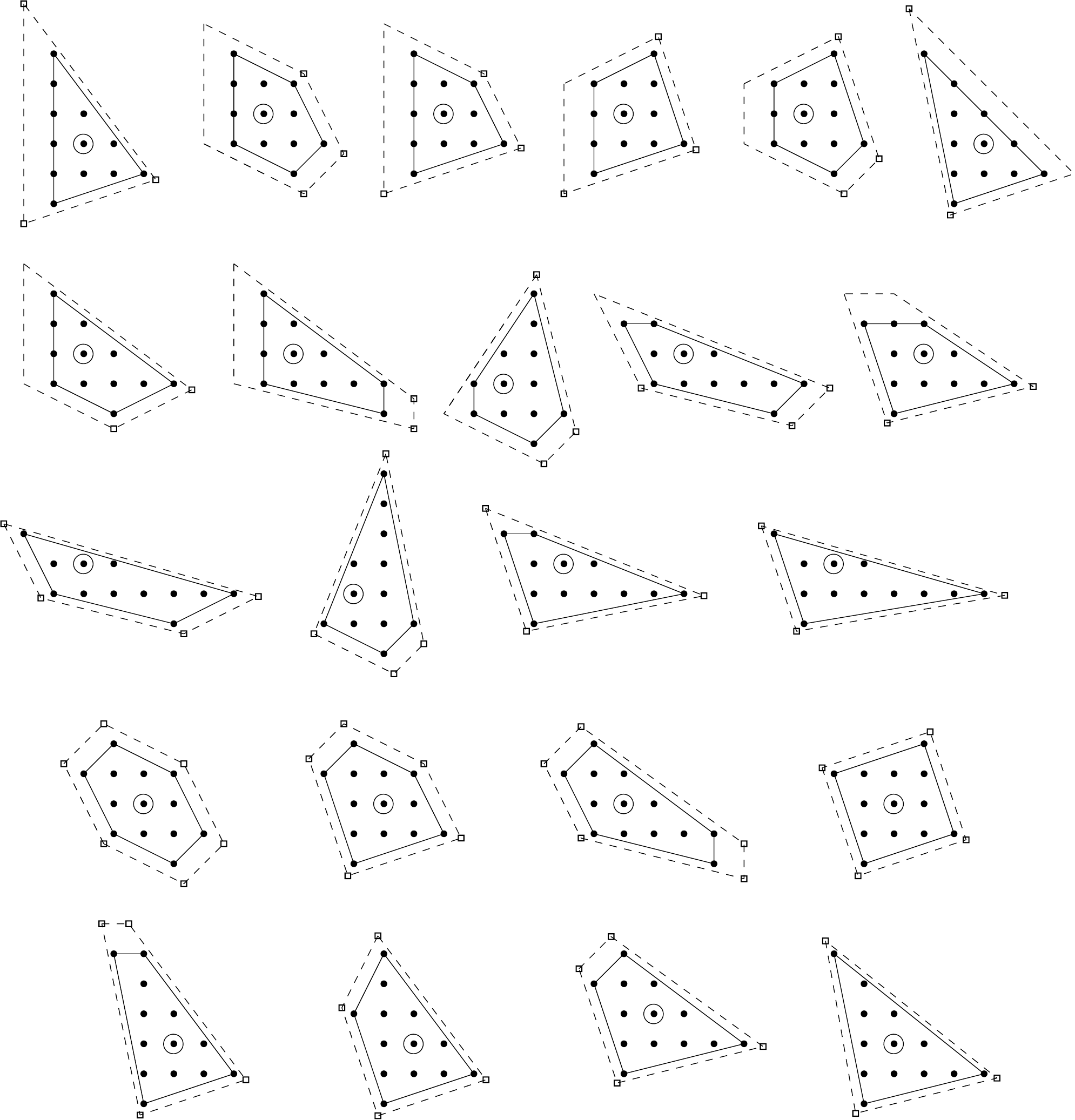}
\caption{All non-hyperelliptic panoptigons with $12$ or $13$ lattice points, along with their relaxed polygons}
\label{figure:panoptigons_12_or_13}

\end{figure}

\end{appendices}

\bibliographystyle{abbrv}

\end{document}